\documentclass[11pt]{article}
\usepackage{amsthm, amsmath, amssymb, amsfonts, url, booktabs, tikz, setspace, fancyhdr, bm, mathrsfs}
\usepackage{hyperref}
\usepackage{geometry}
\usepackage{hyperref, enumerate}
\usepackage[shortlabels]{enumitem}
\usepackage[babel]{microtype}
\usepackage[english]{babel}
\usepackage[capitalise]{cleveref}
\usepackage{comment}
\usepackage{bbm}
\usepackage{csquotes}
\usepackage{mathabx}
\usepackage{tikz}
\usepackage{graphicx}
\usepackage{float}
\usepackage{amsmath}

\counterwithin{figure}{section}

\newtheorem{theorem}{Theorem}[section]

\newtheorem{lemma}[theorem]{Lemma}
\newtheorem{cor}[theorem]{Corollary}

\newtheorem{claim}[theorem]{Claim}

\theoremstyle{definition}
\newtheorem{defn}[theorem]{Definition}
\newtheorem*{defn-non}{Definition}

\newtheorem{ques}[theorem]{Question}
\usepackage[linesnumbered, ruled]{algorithm2e}
\SetKwRepeat{Do}{do}{while}%

\newenvironment{poc}{\begin{proof}[Proof of claim]}{\end{proof}}

\usepackage{todonotes}

\newcommand*{\abs}[1]{\lvert#1\rvert}

\newcommand{\cF}{\mathcal{F}}

\newcommand{\cH}{\mathcal{H}}

\title{Largest $3$-uniform set systems with VC-dimension 2}
\author{
Jian Wang\thanks{Department of Mathematics,
Taiyuan University of Technology,
Taiyuan, China. Email: wangjian01@tyut.edu.cn
}
\and 
Zixiang Xu\thanks{Extremal Combinatorics and Probability Group (ECOPRO), Institute for Basic Science (IBS), Daejeon, South Korea. Email: zixiangxu@ibs.re.kr}
\and
Shengtong Zhang\thanks{Department of Mathematics, Stanford University, USA. Email: stzh1555@stanford.edu}
}

\begin{document}

\maketitle
\begin{abstract}
   We determine the largest size of $3$-uniform set systems on $[n]$ with VC-dimension $2$ for all $n$.
\end{abstract}

\section{Introduction}
Let $\mathcal{F}\subseteq 2^{X}$ be a set system with ground set $X$. Given a subset $S \subseteq X$, we say that $\mathcal{F}$ \emph{shatters} $S$ if $\{F \cap S : F \in \mathcal{F}\} = 2^S$, that is, every subset of $S$ can be realized as the intersection of $S$ with some member of $\mathcal{F}$. The \emph{Vapnik–Chervonenkis dimension} (\emph{VC-dimension} for short) of $\mathcal{F}$ is the largest integer $d$ such that there exists  $S \subseteq X$ with $|S| = d$ that is shattered by $\mathcal{F}$. The VC-dimension is a key measure of the complexity of $\mathcal{F}$, with wide applications in combinatorics, learning theory, and discrete geometry.

As early as the 1970s, three groups of researchers~\cite{1972JCTASauer,1972PACJMShelah,1971TPAVCVC} independently established a fundamental result, which states that any set system \(\mathcal{F} \subseteq 2^{[n]}\) with VC-dimension at most \(d\) satisfies $
|\mathcal{F}| \le \sum_{i=0}^{d} \binom{n}{i}$.
This bound is tight and can be achieved by several distinct extremal constructions. Despite this classical result, the following fundamental question concerning uniform set systems remains open.

\begin{ques}\label{question}
 For given integers \(d \ge 2\) and \(n \ge 2d + 2\), what is the maximum size of $\mathcal{F}\subseteq\binom{[n]}{d+1}$ with VC-dimension at most \(d\)? 
\end{ques}

An elegant upper bound for this problem is \(\binom{n}{d}\), established by Frankl and Pach~\cite{1984Franklpach} using the so-called linear algebra methods in 1984. Over the past four decades, this bound has been improved in various settings. In particular, Ge, Xu, Yip, Zhang, and Zhao~\cite{2024FranklPach} recently proved that \(\binom{n}{d}\) is not tight for every \(d\) and every \(n \ge 2d + 2\). When \(n\) is sufficiently large and \(d\) is a prime power, Mubayi and Zhao~\cite{2007JAC} showed that the Frankl--Pach bound can be improved by an additive term of order \(\Omega_d(\log n)\). The above upper bounds are all based on algebraic methods. Very recently, Chao, Xu, Yip, and Zhang~\cite{2025CombProof} improved the upper bound via a purely combinatorial argument, showing that for any \(d \ge 2\) and sufficiently large \(n\), the size of a \((d+1)\)-uniform set system with VC-dimension \(d\) is at most
\[
\binom{n-1}{d} + O_d\left(n^{d - 1 - \frac{1}{4d - 2}}\right).
\]

On the other hand, a natural lower bound is \(\binom{n-1}{d}\), given by the canonical star construction. Moreover, Erd\H{o}s~\cite{1984Erdos}, Frankl and Pach~\cite{1984Franklpach} conjectured that \(\binom{n-1}{d}\) is the maximum size when $n$ is large. However, in 1997, Ahlswede and Khachatrian~\cite{1997CombFan} provided a surprising counterexample, proving a stronger lower bound \(\binom{n-1}{d} + \binom{n-4}{d-2}.\)
This construction was later generalized by Mubayi and Zhao~\cite{2007JAC}, who showed that there exist many non-isomorphic extremal constructions achieving this lower bound. Based on these findings, Mubayi and Zhao~\cite{2007JAC} conjectured that the lower bound
\(
\binom{n-1}{d} + \binom{n-4}{d-2}
\)
is in fact the solution to~\cref{question} when $n\ge 2(d+1)$.

The main purpose of this paper is to completely determine the maximum size of $3$-uniform set systems on $[n]$ with VC-dimension $2$\footnote{As established by Frankl and Pach in~\cite{1984Franklpach}, any $3$-uniform set system with VC-dimension $1$ has size at most $n < \binom{n-1}{2} + 1$ for $n \ge 7$. Therefore, throughout this paper, the condition of having VC-dimension at most $2$ can typically be understood as having VC-dimension exactly $2$.
} for all $n$, thereby completely resolving~\cref{question} when $d=2$.

When $n \in \{3,4,5\}$, the answer is $\binom{n}{3}$ since $\binom{[n]}{3}$ is intersecting, so it has VC-dimension $2$. To our surprise, the conjecture of Mubayi and Zhao does not hold when $n=6$ and $d=2$. More precisely, there exists a $3$-uniform set system $\mathcal{F}$ on $[6]$ of size $13$ having VC-dimension $2$:
\[
\mathcal{F} = \left\{
\begin{array}{l}
\{1, 2, 6\}, \{1, 3, 5\}, \{1, 4, 5\}, \{1, 4, 6\}, \{1, 5, 6\}, \\
\{2, 3, 4\}, \{2, 4, 5\}, \{2, 4, 6\}, \{2, 5, 6\}, \\
\{3, 4, 5\}, \{3, 4, 6\}, \{3, 5, 6\}, \{4, 5, 6\}
\end{array}
\right\}
\]
On the other hand, an explicit computation using computer programs (see Appendix and~\cite{2025CodesForsmallN}) shows that the size of any $3$-uniform set system on $[6]$ with VC-dimension $2$ is at most $13$. This confirms that the correct answer to~\cref{question} for $n = 6$ and $d = 2$ is $13$.

Our main result reads as follows. 
\begin{theorem} 
\label{thm:VC-2}
    Let $n\ge 7$. If $\mathcal{F}\subseteq\binom{[n]}{3}$ has VC-dimension at most $2$, then $|\mathcal{F}|\le\binom{n-1}{2}+1$.
\end{theorem}
Combine lower bound from the construction in~\cite{1997CombFan,2007JAC}, the largest size of $3$-uniform set system on $[n]$ with VC-dimension at most $2$ for $n\ge 7$ is exactly $\binom{n-1}{2}+1$. This marks an initial advance in tackling Question~\ref{question}, and settles the specific case concerning $3$-uniform set systems, which was explicitly raised by Mubayi and Zhao~\cite{2007JAC} in 2007.

An interesting phenomenon is that the answer to this problem is divided into three distinct cases:
\begin{enumerate}
    \item When \( n \leq 5 \), the answer is \( \binom{n}{3} \). 
    \item When \( n = 6 \), the answer is \( \binom{n - 1}{2} + 3 = 13 \). 
    \item When \( n \geq 7 \), the answer is \( \binom{n - 1}{2} + 1 \).
\end{enumerate}
  This segmentation reveals intriguing structural changes in the problem for different ranges of $n$.

\section{Proof of~\cref{thm:VC-2}}
The proof of~\cref{thm:VC-2} consists of two parts. First, we verify \cref{thm:VC-2} for $n = 7$ via explicit computation, see Appendix and~\cite{2025CodesForsmallN}.
\begin{lemma}\label{lemma:7316}
    If $\mathcal{F}\subseteq\binom{[7]}{3}$ has VC-dimension at most $2$, then $|\mathcal{F}|\le\binom{6}{2}+1= 16$.
\end{lemma}
Then, we complete the proof via induction.
\begin{theorem} \label{Induction}
Let $n \geq 8$. If \cref{thm:VC-2} holds for $(n - 1)$, then it holds for $n$.
\end{theorem}

\subsection{Notations and preliminary structural results}

Note that when \( \mathcal{F} \subseteq \binom{[7]}{3} \) has VC-dimension at most \(2\), by~\cref{lemma:7316}, its size satisfies  
\[
|\mathcal{F}| \le \binom{7 - 1}{2} + 1=16.
\]  
Our proof proceeds by induction on \(n\). Assume that for all integers \(7 \le n' \le n - 1\), any set system \(\mathcal{F}' \subseteq \binom{[n']}{3}\) with VC-dimension at most \(2\) satisfies  
\[
|\mathcal{F}'| \le \binom{n' - 1}{2} + 1.
\]  
Now consider \(\mathcal{F} \subseteq \binom{[n]}{3}\) with VC-dimension at most \(2\). Our goal is to show that  
\[
|\mathcal{F}| \le \binom{n - 1}{2} + 1.
\]

Let \(\mathcal{F} = \{F_1,F_2, \ldots, F_m\} \subseteq \binom{[n]}{3}\) be a set system with VC-dimension at most \(2\). In remainder of the proof, we always assume that
\[
m \ge \binom{n - 1}{2} + 2
\]  
and aim to derive a contraction.

By the definition of VC-dimension, for each \(F_i\in\mathcal{F}\) (with \(i\in [m]\)) there exists at least one proper subset \(B_i\subsetneq F_i\) such that  \(F\cap F_i\neq B_i\) for every \(F\in\mathcal{F}\). Henceforth, we assume that for each \(i\in [m]\),  \(B_i\) is chosen to have the largest cardinality among all available options. Usually we refer to this set $B_{i}$ as the \emph{witness} of $F_{i}$. In~\cite{2025CombProof}, the following auxiliary result was established using the sunflower lemma~\cite{2024BVCSunflower,1960Sunflower, 2023BVCSunflower}, and the proof was inspired by~\cite[Lemma 3]{2007JAC} 

\begin{lemma}[\cite{2025CombProof}]\label{lemma:2222}
Let \(\mathcal{F}=\{F_{1},F_{2},\ldots,F_{m}\}\subseteq \binom{[n]}{d}\) be a set system with VC-dimension at most \(d\) and \(\{B_{1},\ldots,B_{m}\}\) be selected as above. Then the followings hold.
\begin{enumerate}
    \item[\textup{(1)}] For each subset $A\subseteq [n]$, the number of indices $i\in [m]$ with $B_{i}=A$ is at most $C_d$, where $C_d$ is a constant. Consequently, for each integer $0\le s\le d$, the number of indices $i\in [m]$ with $|B_{i}|=s$ is at most $C_d \cdot\binom{n}{s}$.
    \item[\textup{(2)}] In particular, when $|A|=d$, the number of indices $i\in [m]$ with $B_{i}=A$ is at most one.
\end{enumerate}
\end{lemma}

For each \(x \in [n]\), define  
\[
\mathcal{F}(x) := \{F\in\mathcal{F} : x \in F\}
\]  
and  
\[
\mathcal{F}(\bar{x}) := \{F\in\mathcal{F}:x\notin F\}.
\]
We also define
\[\mathcal{G}(x):=\{F\setminus\{x\}:F\in\mathcal{F}(x)\}.\]
Clearly, $|\mathcal{F}(x)|=|\mathcal{G}(x)|$ for any $x\in [n]$. Similarly, for distinct $x,y\in [n]$, we define 
\[
\cF(x,y) = \{F\in\mathcal{F} : \{x,y\}\subseteq F\}. 
\]

We begin with a simple observation derived from our assumptions, which serves as a key density condition throughout the proof.
\begin{claim}\label{claim:LargeDegree}
    For any $x\in [n]$, $|\mathcal{F}(x)|\ge n-1$.
\end{claim}
\begin{poc}
   Suppose that \( |\mathcal{F}(x)| \le n - 2 \) for some \( x \in [n] \). Since \( \mathcal{F}(\bar{x}) \subseteq \binom{[n] \setminus \{x\}}{3} \) has VC-dimension at most \(2\), the inductive hypothesis gives  
\[
|\mathcal{F}(\bar{x})| \le \binom{n - 2}{2} + 1.
\]  
Therefore,  
\[
|\mathcal{F}| = |\mathcal{F}(x)| + |\mathcal{F}(\bar{x})| \le (n - 2) + \left( \binom{n - 2}{2} + 1 \right) = \binom{n - 1}{2} + 1,
\]
a contradiction.
\end{poc}

Our proof makes essential use of the following three structures.

\begin{defn}\label{def:structure    }
    Let $\mathcal{F}=\{F_{1},F_{2},\ldots,F_{m}\}\subseteq\binom{[n]}{3}$ be a set system with VC-dimension at most $2$ and $\{B_{1},B_{2},\dots,B_{m}\}$ be selected as above. We define the following set families.
    \begin{enumerate}
        \item[\textup{(1)}] Define $\mathcal{B}$ as the collection of witness sets of size two:  \[
\mathcal{B} := \{ B_i : i \in [m],\ |B_i| = 2 \}.
\]  
\item[\textup{(2)}] Define $L$ as the set of elements in $[n]$ which serves as singleton witness: \[L=\bigg\{x\in [n]:\textup{there\ exists\ some\ }i\in [m]\ \textup{with\ }B_{i}=\{x\}\bigg\}.\]

\item[\textup{(3)}] Define $\mathcal{C}$ as the collection of sets in $\mathcal{F}$ whose witness is empty set: \[\mathcal{C}:=\{F_{i}\in\mathcal{F}:B_{i}=\emptyset\}.\]
    \end{enumerate}
\end{defn}
Since $\mathcal{B}$ is $2$-uniform, we can regard it as a graph, whose vertex set is $[n]$ and edge set is $\mathcal{B}$. Moreover by~\cref{lemma:2222}(2), \(\mathcal{B}\) is a simple graph. As is commonly used in graph theory, we denote \( N_{\mathcal{B}}(x) \) as the set of neighbors of the vertex \( x \) in the graph \( \mathcal{B} \).

Based on~\cref{claim:LargeDegree}, we have the following observation.
\begin{claim}\label{claim:LowDegreeInB}
    For any $x\in L$, $|N_{\mathcal{B}}(x)|\le n-4$.  
\end{claim}
\begin{poc}
  Recall that $\mathcal{G}(x):=\{F\setminus\{x\}:F\in\mathcal{F}(x)\}$, we first observe that the transversal number of \(\mathcal{G}(x)\) is at most \(2\). Indeed, let \(F_{i} = \{x, y, z\}\) correspond to \(B_{i} = \{x\}\). By the definition of \(B_{i}\), the set \(\{y, z\}\) is a transversal of \(\mathcal{G}(x)\). Since \(|\mathcal{G}(x)|=|\mathcal{F}(x)| \ge n - 1\), there exists some \(w \in [n] \setminus \{x, y, z\}\) such that both \(\{w, y\}\) and \(\{w, z\}\) belong to \(\mathcal{G}(x)\). In other words,  
\[
\{x, y, z\}, \quad \{x, y, w\}, \quad \text{and} \quad \{x, z, w\} \in \mathcal{F}.
\]  
Therefore,  
\[
\{x, y\}, \quad \{x, z\}, \quad \text{and} \quad \{x, w\} \notin \mathcal{B},
\]  
which implies  
\[
|N_{\mathcal{B}}(x)| \le n - 4.
\]
\end{poc}

\subsection{Technical overview}
The proof of the main theorem proceeds by induction on \(n\), with the base case \(n = 7\) verified by explicit computation. The inductive step (see~\cref{Induction}) assumes the upper bound holds for \(n - 1\), and aims to derive a contradiction if a 3-uniform set system \(\mathcal{F} \subseteq \binom{[n]}{3}\) with VC-dimension at most 2 satisfies \(|\mathcal{F}| \ge \binom{n-1}{2} + 2\). While the proof is self-contained and elementary, it involves a considerable amount of detailed case analysis. To help the reader better understand the proof, we provide here an overview of the entire argument, and highlight several auxiliary results that we consider important.

The main strategy relies on a sophisticated structural analysis. For each \(F_i \in \mathcal{F}\), a witness subset \(B_i \subsetneq F_i\) is chosen as large as possible such that no other set in \(\mathcal{F}\) intersects \(F_i\) exactly in \(B_i\). By~\cref{lemma:2222} (based on prior work~\cite{2025CombProof} using the sunflower lemma), the number of times a given \(B_i\) can appear is bounded. In particular, the collection of 2-element witnesses \(\{B_i : |B_i| = 2\}\) forms a simple graph \(\mathcal{B}\).

The first key technical insight is~\cref{claim:LargeDegree}, which asserts that every element must appear in at least \(n - 1\) sets of \(\mathcal{F}\), enforcing a strong density condition. This motivates the introduction of several structural tools: the auxiliary graph \(\mathcal{B}\), the set \(L\) of vertices appearing as singleton witnesses \(B_i = \{x\}\), along with the collection \(\mathcal{C}\) of sets \(F_i\) with empty witnesses \(B_i = \emptyset\).

A central component of the proof is~\cref{lemma: NoTwoB1}, which shows that no element can occur in more than one singleton witness \(B_i = \{x\}\). Its proof is intricate and builds on lots of observations. A key outcome of this lemma is the concise identity given in~\cref{cor:MainInequality}:
\[
|\mathcal{F}| = |\mathcal{B}| + |L| + |\mathcal{C}|,
\]
which gives us tighter control on the size of \(\mathcal{F}\). The remainder of the proof is divided into two main cases according to whether $\mathcal{C}$ is empty. 

In~\cref{subsection:CisEmpty}, we analyze the case $\mathcal{C} = \emptyset$. Here we want to understand the structural of the graph $\mathcal{B}$ formed by witness sets of size two. Using the size constraint on the links $\cF(x)$, we show that $\mathcal{B}$ cannot be too dense. We approach this by examining the maximum degree $\Delta(\mathcal{B})$ and consider the cases $\Delta = n-2$, $n-3$, and $n-4$. Each of these cases requires tailored arguments based on neighborhood properties, set intersections, and auxiliary counting estimates. At the end of this analysis we can already prove that
$$\abs{\cF} \leq \binom{n - 1}{2} + 1 + \abs{\mathcal{C}} \leq \binom{n - 1}{2} + O(1).$$
However, lots of work remain to reduce the $O(1)$ to $1$. To achieve this, we perform a sophisticated structural analysis on $\mathcal{C}$.

In~\cref{subsection:StructralofC}, we analyze interactions among $\mathcal{B}$, $L$, and $\mathcal{C}$ when $\mathcal{C}\neq\emptyset$. Key claims like~\cref{claim:MissTwoedges} and~\cref{claim:XYZLE2} quantify precisely how $\mathcal{C}$ restricts the structure and size of $\mathcal{B}$ and $L$. We also systematically identify and exclude forbidden substructures in $\mathcal{C}$, such as 2-intersecting families and linear triangle in~\cref{claim:Not2Intersection} and~\cref{claim:NoLinearTriangle}, thus significantly restricting the structure of $\mathcal{C}$.

In \cref{subsection:Transversalof C}, we analyze the structure of $\mathcal{C}$ based on its transversal number $\tau(\mathcal{C})$. Noting that $\mathcal{C}$ is intersecting, its transversal number is at most 3. We first quickly eliminate the case where the transversal number is exactly $3$, as this implies $|\mathcal{F}| \leq 19$.

The case $\tau(\mathcal{C})=2$ is involved. Here, two earlier structural properties in~\cref{claim:Not2Intersection} and~\cref{claim:NoLinearTriangle} become essential. In particular, we leverage the density condition from~\cref{claim:LargeDegree} and the greedy construction of linear triangles to derive contradictions. These tools allow us to rule out this scenario by consistently constructing forbidden substructures under the given constraints.

The most intricate case arises when $\tau(\mathcal{C})=1$. In this case, $\mathcal{C}$ must share a common element. However, its structure remains highly complex. To handle this, we draw on the density condition and the interplay among $\mathcal{B}$, $L$, and $\mathcal{C}$, as developed in~\cref{subsection:StructralofC}. We also exploit structural features of $\mathcal{C}$ itself. Crucially, we shift focus to the link graph $\mathcal{Z} = \{F \setminus \{z\} : F \in \mathcal{C}\subseteq\mathcal{F}\}$ for the center $z$ of the star. By analyzing several combinatorial properties of $\mathcal{Z}$ in~\cref{claim:PropertyZ}, including its matching number, maximum degree, and the absence of $K_{2,2}$ subgraphs, we efficiently complete the argument, avoiding exhaustive enumeration of all possible $\mathcal{Z}$'s.

\subsection{Key lemma: Structural insights of $L$}
By~\cref{lemma:2222}(1), we know that for any element \(x \in L\), there are at most \(C\) sets in \(\mathcal{F}\) whose witness is exactly \(\{x\}\), for some constant \(C > 1\). The main result of this section is~\cref{lemma: NoTwoB1}, where we improve this bound and show that the constant $C$ can be strengthened to 1. Before proving this, we begin with a preliminary step showing $C \leq 2$.
\begin{claim}\label{claim:SizeofB1}
    For each $x\in [n]$, $|\{\ell\in [m]:B_{\ell}=\{x\}\}|\le 2$. Moreover, the equality holds if and only if
    \[\mathcal{F}(x)=\big\{\{x,y,z\},\{x,y,w\},\{x,z,w\}\big\}\cup\bigg(\bigcup_{v\in [n]\setminus\{x,y,z,w\}}\{x,y,v\}\bigg),\]
    for some distinct $y, z, w \in [n] \backslash \{x\}$.
\end{claim}
\begin{poc}
  Suppose there exist \(i, j, k \in [m]\) such that  
\(
B_{i} = B_{j} = B_{k} = \{x\}.
\) Let \(F_{i} = \{x, y, z\}\), then the set \(\{y, z\}\) is a transversal of \(\mathcal{G}(x)\). By~\cref{claim:LargeDegree}, \(|\mathcal{G}(x)|=|\mathcal{F}(x)| \ge n - 1\), there exists some \(w \in [n] \setminus \{x, y, z\}\) such that $\{y,w\},\{z,w\}\in\mathcal{G}(x)$, which also yields that 
\[
\{x, y, w\},\ \{x, z, w\} \in \mathcal{F}(x).
\]  
Then \(F_{j}\) and \(F_{k}\) must be exactly \(\{x, y, w\}\) and \(\{x, z, w\}\), respectively. Otherwise, \(F_{j}\cap \{x, y, w\} \neq \{x\}\) or \(F_{k} \cap \{x, z, w\} \neq \{x\}\), contradicting the assumption that \(B_{j} = B_{k} = \{x\}\). Thus, any set in $F\in\mathcal{F}(x)$ must satisfies $|F\cap\{y,z,w\}|\ge 2$, which implies that \(
|\mathcal{F}(x)| = 3,\) contradicting the assumption that \( |\mathcal{F}(x)| \ge n - 1 \). Therefore, \(\left| \{i \in [m] : B_{i} = \{x\} \} \right| \le 2.\)

Furthermore, suppose that \(\left| \{\ell \in [m] : B_{\ell} = \{x\} \} \right| = 2\) for some \(x \in [n]\). Let \(F_i = \{x, y, z\} \in \mathcal{F}\) be one of the sets with witness \(B_i = \{x\}\). By the preceding argument, we may assume without loss of generality that the other such set is \(F_j = \{x, y, w\}\) with \(B_j = \{x\}\). Then every set \(F \in \mathcal{F}(x) \setminus \{ \{x, y, z\}, \{x, y, w\}, \{x, z, w\} \}\) must contain both \(x\) and \(y\); otherwise, its intersection with either \(\{x, y, z\}\) or \(\{x, y, w\}\) would be exactly \(\{x\}\), contradicting the definition of a witness.

Moreover, since \(|\mathcal{F}(x)| \ge n - 1\) by~\cref{claim:LargeDegree}, we conclude that \(\mathcal{F}(x)\) must be exactly  
\[
\big\{ \{x, y, z\}, \{x, y, w\}, \{x, z, w\} \big\} \cup \bigg( \bigcup_{v \in [n] \setminus \{x, y, z, w\}} \{x, y, v\} \bigg).
\]
\end{poc}
We now state the central lemma of this subsection. It asserts that, under the assumption \( |\mathcal{F}| \ge \binom{n - 1}{2} + 2 \) and the inductive hypothesis, for any fixed \(x \in [n]\), there is at most one index \(i \in [m]\) such that \(B_i = \{x\}\).
\begin{lemma}\label{lemma: NoTwoB1}
   Let $\mathcal{F}=\{F_{1},\ldots,F_{m}\}\subseteq\binom{[n]}{3}$ be a set system with VC-dimension at most $2$ and $\{B_{i}\}_{i=1}^{m}$ be selected based on the above rules. If $m\ge\binom{n-1}{2}+2$, then for any $x\in [n]$, we have \(\left| \{i \in [m] : B_{i} = \{x\} \} \right| \le 1.\)
\end{lemma}
\begin{proof}[Proof of~\cref{lemma: NoTwoB1}]
 Suppose there is some $x\in [n]$ such that \(\left| \{i \in [m] : B_{i} = \{x\} \} \right| = 2.\) By~\cref{claim:SizeofB1}, there exist elements $y,z,w\in [n]\setminus\{x\}$ such that $F_{i}=\{x,y,z\}$ and $F_{j}=\{x,y,w\}$ satisfy $B_{i}=B_{j}=\{x\}$, and
 \[\mathcal{F}(x)=\big\{\{x,y,z\},\{x,y,w\},\{x,z,w\}\big\}\cup\bigg\{\{x,y,v\}: v \in [n]\setminus\{x,y,z,w\}\bigg\}.\]
 
 We now prove a key claim.
 \begin{claim}\label{claim:yzw}
     $\{y,z,w\}$ belongs to $\mathcal{F}$.
 \end{claim}
 \begin{poc}
 Suppose $\{y,z,w\}\notin\mathcal{F}$, we claim that the set system
 \[\mathcal{H}:=\big(\mathcal{F}\setminus\mathcal{F}(x)\big)\cup\{y,z,w\}\subseteq\binom{[n]\setminus\{x\}}{3}\]
has VC-dimension at most $2$. We verify this as follows.
\begin{enumerate}
    \item Suppose there exists some \(F_{k} \in \mathcal{F} \setminus \mathcal{F}(x)\) such that \(F_{k} \cap \{y, z, w\} = B_{k}.\) Then \(B_{k}\) must be contained in some \(C \in \big\{ \{y, z\}, \{y, w\}, \{z, w\} \big\}\). On the other hand, we already know that  
\[
\{x, y, z\}, \quad \{x, y, w\}, \quad \text{and} \quad \{x, z, w\} \in \mathcal{F}(x) \subseteq \mathcal{F}.
\]  
Therefore,  
\[
F_{k} \cap (C \cup \{x\}) = F_{k} \cap C = B_{k},
\]  
which contradicts the choice of \(B_{k}\). This shows that any set in $\mathcal{F}\setminus\mathcal{F}(x)$ is not shattered by $\cH$.

    \item It remains to show that \(\{y, z, w\}\) is not a shattered set in \(\mathcal{H}\). Let  
\[
F_{\ell} = \{x, z, w\} \in \mathcal{F}(x) \subseteq \mathcal{F}.
\]  
Moreover, observe that  
\[
\{x\} = \{x, z, w\} \cap \{x, y, v\}, \quad \text{for some } v \in [n] \setminus \{x, y, z, w\},
\]  
\[
\{x, z\} = \{x, z, w\} \cap \{x, y, z\},
\]  
and  
\[
\{x, w\} = \{x, z, w\} \cap \{x, y, w\}.
\]  
These intersections imply that
\begin{equation}\label{eq:Bell}
    B_{\ell} \subseteq \{z, w\}.
\end{equation}
Now suppose there exists \(F \in \mathcal{F} \setminus \mathcal{F}(x)\) such that  
\[
F \cap \{y, z, w\} = B_{\ell}.
\]  
Then we would have  
\[
F \cap \{x, z, w\} = B_{\ell},
\]  
contradicting the choice of \(B_{\ell}\).
\end{enumerate}
Since $\mathcal{H}\subseteq\binom{[n]\setminus\{x\}}{3}$ has VC-dimension at most $2$, by inductive hypothesis, we conclude that
\[|\mathcal{F}|=|\mathcal{H}|+|\mathcal{F}(x)|-1\le\binom{n-2}{2}+1+n-1-1=\binom{n-1}{2}+1,\]
 a contradiction to $|\mathcal{F}|\ge\binom{n-1}{2}+2$. 
 \end{poc}
By~\cref{claim:yzw}, we have 
 \[
\big\{ \{x, y, z\}, \{x, y, w\}, \{x, z, w\},\{y,z,w\} \big\} \cup \bigg( \bigcup_{v \in [n] \setminus \{x, y, z, w\}} \{x, y, v\} \bigg)\subseteq\mathcal{F}.
\] 
Our next goal is to determine the precise structure of $\mathcal{F}(x)\cup\mathcal{F}(z)$.
\begin{claim}\label{claim:StructureFxFz}
    Swapping the role of $z$ and $w$ if necessary, we have
   \[\mathcal{F}(x)\cup\mathcal{F}(z)=\big\{ \{x, y, z\}, \{x, y, w\}, \{x, z, w\},\{y,z,w\} \big\} \cup \bigg( \bigcup_{v \in [n] \setminus \{x, y, z, w\}} \big(\{x, y, v\}\cup\{z,w,v\}\big) \bigg).\]
\end{claim}
\begin{poc}
Recall that for \(F_{\ell} = \{x, z, w\}\), we have \(B_{\ell} \subseteq \{z, w\}\), as established in~\eqref{eq:Bell} of the proof of~\cref{claim:yzw}. Also note that \(\{x, z, w\} \cap \{y, z, w\} = \{z, w\}\), therefore $B_{\ell}\neq\{z,w\}$.

We first consider the case that \(B_{\ell} \neq \emptyset\), it then follows that \(B_{\ell}\) is either \(\{z\}\) or \(\{w\}\). By symmetry, we may assume \(B_{\ell} = \{z\}\). Consider the set system  
\[
\mathcal{G}(z) := \{F \setminus \{z\} : F \in \mathcal{F}(z)\}.
\]  
By our assumption, \(\mathcal{G}(z)\) can be covered by \(\{x, w\}\). Since \(F_{j} = \{x, y, w\}\) satisfies \(B_{j} = \{x\}\), there is no other set \(F \in \mathcal{F}\) such that  
\[
F \cap F_{j} = \{x\}.
\]  
Therefore, the only set in \(\mathcal{F}(z)\) containing \(x\) is \(F_{i} = \{x, y, z\}\). Moreover, since \(|\mathcal{F}(z)| \ge n - 1\) by~\cref{claim:LargeDegree}, we conclude that \(\mathcal{F}(z)\) has the following structure:  
\[
\mathcal{F}(z) = \{x, y, z\} \cup \bigg( \bigcup_{v \in [n] \setminus \{x, w\}} \{x, w, v\} \bigg).
\]
This further yields
\[
\mathcal{F}(x)\cup\mathcal{F}(z)=\big\{ \{x, y, z\}, \{x, y, w\}, \{x, z, w\},\{y,z,w\} \big\} \cup \bigg( \bigcup_{v \in [n] \setminus \{x, y, z, w\}} \big(\{x, y, v\}\cup\{z,w,v\}\big) \bigg).
\]

It remains to consider the case where \( B_{\ell} = \emptyset \), which yields that \( F_{\ell} = \{x, z, w\} \) is a transversal set of \( \mathcal{F} \). Moreover, recall that by~\cref{claim:SizeofB1}
\[
\mathcal{F}(x) = \Big\{ \{x, y, z\}, \{x, y, w\}, \{x, z, w\} \Big\} \cup \Bigg( \bigcup_{v \in [n] \setminus \{x, y, z, w\}} \{x, y, v\} \Bigg).
\]

We next show that there is no element \( v \in [n]\setminus \{x, y, z, w\} \) and index \( p \in [m] \) such that \( B_p = \{v\} \). Suppose \( B_p = \{v\} \) for some \( p \in [m] \) and \( v \in [n]\setminus \{x, y, z, w\} \). By~\cref{claim:LowDegreeInB}, there exist distinct elements \( a, b, c \in [n] \setminus \{v\} \) such that
\[
\{a, b, v\}, \quad \{a, c, v\}, \quad \{b, c, v\} \in \mathcal{F}(v).
\]
Observe that \( x \notin \{a, b, c\} \), since \( \{x, y, v\} \) is the only set in \( \mathcal{F} \) containing both \( x \) and \( v \). On the other hand, \( \{z, w\} \subseteq \{a, b, c\} \); without loss of generality, assume \( b = z \) and \( c = w \). However, one can see that
\begin{enumerate}
    \item $\{z,w,v\}\cap\{z,v,a\}=\{z,v\}$, $\{z,w,v\}\cap\{w,v,a\}=\{w,v\}$ and $\{z,w,v\}\cap \{z,w,a\}=\{z,w\}$;
    \item $\{z,w,v\}\cap \{x,y,z\}=\{z\}$, $\{z,w,v\}\cap\{x,y,w\}=\{w\}$ and $\{z,w,v\}\cap \{x,y,v\}=\{v\}$;
    \item $\{z,w,v\}\cap\{x,y,v'\}=\emptyset$ for some $v'\in [n]\setminus\{x,y,z,w,v\}$, note that such $v'$ exists since $n\ge 6$.
\end{enumerate}
This implies that $\{z,w,v\}$ is a shattered set of $\mathcal{F}$, a contradiction. Therefore, there is no element $v\in [n]\setminus \{x,y,z,w\}$ and index $p\in [m]$ such that $B_{p}=\{v\}$.
Recall
\[
\mathcal{B} := \{ B_i : i \in [m],\ |B_i| = 2 \}.
\]  
We claim that for each \( v \in [n] \setminus \{x, y, z, w\} \), at least one of \( \{z, v\} \) and \( \{w, v\} \) is not in \( \mathcal{B} \). Suppose, for contradiction, that both \( \{z, v\} \) and \( \{w, v\} \) are in \( \mathcal{B} \). Then there is at most one set in \( \mathcal{F} \) containing \( \{z, v\} \), and at most one set containing \( \{w, v\} \). Moreover, by~\cref{claim:LowDegreeInB}, we know that \( \{x, y, v\} \) is the only set in \( \mathcal{F}(x) \) containing \( \{x, v\} \). Since \( \{x, z, w\} \) is a transversal of \( \mathcal{F} \), it follows that \( |\mathcal{F}(v)| \le 3 \), contradicting~\cref{claim:LargeDegree}.

Recall that  
\[
\Big\{ \{x, y, z\}, \{x, y, w\}, \{x, z, w\}, \{y, z, w\} \Big\} \subseteq \mathcal{F},
\]  
which implies that none of the \( \binom{4}{2} = 6 \) pairs in \( \{x, y, z, w\} \) belong to \( \mathcal{B} \). Moreover, for each \( v \in [n] \setminus \{x, y, z, w\} \), since \( \{x, y, v\} \in \mathcal{F} \), at most one of \( \{x, v\} \) and \( \{y, v\} \) can belong to \( \mathcal{B} \). Combining these observations, we obtain  
\[
|\mathcal{B}| \le \binom{n}{2} - \Big( 6 + (n - 4) + (n - 4) \Big) = \binom{n}{2} - (2n - 2).
\]
In addition, suppose $v$ satisfies that $B_{p}=\{v\}$ for some $p\in [m]$, then $v\in\{x,y,z,w\}$, and there are at most two such indices by~\cref{claim:SizeofB1}. First we can check $v\neq y$ by the following argument based on the fact that $\{x,z,w\}$ is a transversal of $\mathcal{F}$.
\begin{enumerate}
    \item If $F_{p}$ does not contain $x$, then $F_{p}$ is of the form either $\{y,z,v_{1}\}$ or $\{y,w,v_{1}\}$ for some $v_{1}\in [n]\setminus\{x,y\}$, which further implies that $\{x,y,v'\}\cap F_{p}=\{y\}$ for some $v'\in [n]\setminus\{x,y,z,w,v_{1}\}$. Note that such $v'$ exists since $n\ge 6$.
    \item If $F_{p}\in \mathcal{F}(x)$, since for $F_{i}=\{x,y,z\}$ and $F_{j}=\{x,y,w\}$, we already have $B_{i}=B_{j}=\{x\}$. Thus if $F_{p}=\{x,y,v''\}$ for some $v''\in [n]\setminus\{v,y,z,w\}$, then $\{y,z,w\}\cap F_{p}=\{y\}$.
\end{enumerate}
Second, if \( v = z \), then  
\[
|F_{p} \cap S| \ge 2 \quad \text{for any} \quad S \in \Big\{ \{x, y, z\}, \{x, z, w\}, \{y, z, w\} \Big\} \subseteq \mathcal{F}.
\]  
Since the witness \( B_{i} \) of \( F_{i} = \{x, y, z\} \) is \( \{x\} \), it follows that \( F_{p} \) must be \( \{y, z, w\} \) if \( B_{p} = \{z\} \). Similarly, if \( v = w \), then  
\[
|F_{p} \cap T| \ge 2 \quad \text{for any} \quad T \in \Big\{ \{x, y, w\}, \{x, z, w\}, \{y, z, w\} \Big\} \subseteq \mathcal{F}.
\]  
Since the witness \( B_{j} \) of \( F_{j} = \{x, y, w\} \) is \( \{x\} \), we conclude that \( F_{p} \) must also be \( \{y, z, w\} \) if \( B_{p} = \{w\} \). Therefore,  
\[
\Big\{ F_{p} \in \mathcal{F} : |B_{p}| = 1 \Big\} \subseteq \Big\{ \{x, y, z\}, \{x, y, w\}, \{y, z, w\} \Big\}.
\]

We now claim that  
\[
\left| \left\{ F_q \in \mathcal{F} : B_q = \emptyset \right\} \cup \left\{ F_p \in \mathcal{F} : |B_p| = 1 \right\} \right| \le 8.
\]  
To verify this, recall that \(F_{\ell} = \{x, z, w\}\) corresponds to \(B_{\ell} = \emptyset\), which implies that \(\{x, z, w\}\) covers all sets in \(\mathcal{F}\). We now examine the structure of \(\left\{ F_q \in \mathcal{F} : B_q = \emptyset \right\}\).

\begin{enumerate}
    \item Suppose that \(x \in F_q\). If \(y \notin F_q\), then \(F_q\) must contain at least one of \(z\) or \(w\), since \(\{y, z, w\} \in \mathcal{F}\). By symmetry, we may assume that \(F_q\) contains \(z\). Since \(B_{\ell} = \emptyset\), the maximality of witness \(B_i\) guarantees the existence of some \(F \in \mathcal{F}\) such that \(F \cap \{x, z, w\} = \{w\}\). It follows that \(F \cap F_q = \{w\}\), which implies \(F_q = \{x, z, w\}\).

On the other hand, if \(y \in F_q\), we may write \(F_q = \{x, y, v\}\). By maximality of witness, there exists some \(F \in \mathcal{F}(x)\) such that \(\{x, v\} \subseteq F\) but \(y \notin F\). By the structure of \(\mathcal{F}(x)\) established in~\cref{claim:SizeofB1}, the element \(v\) can only be \(z\) or \(w\). However, for \(F_i = \{x, y, z\}\) and \(F_j = \{x, y, w\}\), we have \(B_i = B_j = \{x\}\), which contradicts the assumption of $B_{q}=\emptyset$. 

Therefore, we conclude that the only possible choice for \(F_q\) containing \(x\) is \(F_q = \{x, z, w\}\).

\item 
If \(x \notin F_q\), then at least one of \(z\) or \(w\) must belong to \(F_q\). By symmetry, assume \(z \in F_q\). Then \(y \in F_q\); otherwise, \(\{x, y, v\} \in \mathcal{F}\) for any \(v \in [n] \setminus \{x, y, z\}\), contradicting the assumption that \(F_q\) covers all sets in \(\mathcal{F}\).

Now suppose there exist three distinct sets  
\[
F_{q_1} = \{y, z, v_1\}, \quad F_{q_2} = \{y, z, v_2\}, \quad F_{q_3} = \{y, z, v_3\}
\]  
with \(v_1, v_2, v_3 \in [n] \setminus \{x, w\}\), and \(B_{q_1} = B_{q_2} = B_{q_3} = \emptyset\). By the maximality of \(B_{q_1}\), there exists \(F \in \mathcal{F}\) such that  
\[
F \cap F_{q_1} = \{v_1\}.
\]  
Since \(F\) does not contain \(y\) or \(z\), we must have  
\[
F = \{v_1, v_2, v_3\}.
\]  
However, this implies \(F \cap \{x, z, w\} = \emptyset\), contradicting the fact that \(\{x, z, w\}\) intersects every set in \(\mathcal{F}\). Thus, the number of possible \(F_q\) of the form \(\{y, z, v\}\), with \(v \in [n] \setminus \{x, w\}\), is at most \(2\). By an analogous argument, the number of possible \(F_q\) of the form \(\{y, w, v\}\), with \(v \in [n] \setminus \{x, z\}\), is also at most \(2\). 
\end{enumerate}
In addition, \(F_q\) might be \(\{y, z, w\}\). Since we have shown that  
\[
\left\{ F_p \in \mathcal{F} : |B_p| = 1 \right\} \subseteq \left\{ \{x, y, z\}, \{x, y, w\}, \{y, z, w\} \right\},
\]  
and from the analysis above regarding \(\left\{ F_q \in \mathcal{F} : B_q = \emptyset \right\}\), we conclude that  
\[
\left| \left\{ F_q \in \mathcal{F} : B_q = \emptyset \right\} \cup \left\{ F_p \in \mathcal{F} : |B_p| = 1 \right\} \right| \le 8.
\]
Since \(n \ge 8\), we have  
\[
|\mathcal{F}| \le |\mathcal{B}| + 8 \le \binom{n}{2} - (2n - 2) + 8 = \binom{n-1}{2} +1 - n + 8 \le \binom{n - 1}{2} + 1,
\]  
which contradicts the assumption that \(|\mathcal{F}| \ge \binom{n - 1}{2} + 2.\) This completes the proof.

\end{poc}
Our next observation is also crucial.
\begin{claim}\label{claim:SmallIntersection}
    If \(F_{h} \in \mathcal{F}\) satisfies \(|F_{h} \cap \{x, y, z, w\}| \le 1,\) then \(B_{h}\in\binom{[n]\setminus\{x,y,z,w\}}{2}.\)
\end{claim}
\begin{poc}
First, \(B_{h} \neq \emptyset\), since all four \(3\)-element subsets of \(\{x, y, z, w\}\) belong to \(\mathcal{F}\), and at least one of them does not intersect \(F_{h}\). Similarly, we can also show that \(B_{h}\) cannot be \(\{v\}\) for any \(v \in \{x, y, z, w\}\). To see this, suppose \(B_{h} = \{v\}\) for some \(v \in [n] \setminus \{x, y, z, w\}\). Since  
\[
F_{h} \cap \{x, y, v\} \neq \{v\},
\]  
\(F_{h}\) must contain at least one of \(x\) or \(y\). Similarly, \(F_{h} \cap \{z, w, v\} \neq \{v\},\) implies that \(F_{h}\) must contain at least one of \(z\) or \(w\). This contradicts the assumption that  
\[
|F_{h} \cap \{x, y, z, w\}| \le 1.
\]
Therefore, for any \(F_{h} \in \mathcal{F}\) with \(|F_{h} \cap \{x, y, z, w\}| \le 1,\) we have \(|B_{h}| = 2\). Furthermore, by the structure of \(\mathcal{F}(x) \cup \mathcal{F}(z)\), it follows that  
\[
F_{h} \notin \mathcal{F}(x) \cup \mathcal{F}(z),
\]  
and hence \(B_{h} \cap \{x, z\} = \emptyset.\) Moreover, if \(B_{h} \cap \{y, w\} \neq \emptyset,\) then there exists some set \(F \in \mathcal{F}(x) \cup \mathcal{F}(z)\) such that  
\[
F \cap F_{h} = B_{h},
\]  
which contradicts the choice of \(B_{h}\). Therefore, we conclude that if \(|F_{h} \cap \{x, y, z, w\}| \le 1\), then 
\[
B_{h} \in \binom{[n] \setminus \{x, y, z, w\}}{2}.
\]  
\end{poc}
We then finish the proof of~\cref{lemma: NoTwoB1} based on the above series of claims. Since for each \(B_{h}\) with \(|B_{h}| = 2\), there is at most one corresponding \(F_{h}\), then by~\cref{claim:SmallIntersection}, we have  
\[
\left| \left\{ F_{h} \in \mathcal{F} : |F_{h} \cap \{x, y, z, w\}| \le 1 \right\} \right| \le \binom{n - 4}{2}.
\]

On the other hand, if \(|F_{h} \cap \{x, y, z, w\}| \ge 2\), then either \(F_{h} \in \mathcal{F}(x) \cup \mathcal{F}(z)\), or \(F_{h}\) is of the form \(\{y, w, v\}\) for some \(v \in [n] \setminus \{x, y, z, w\}\). Therefore,  
\[
|\mathcal{F}| \le \binom{n - 4}{2} + |\mathcal{F}(x) \cup \mathcal{F}(z)| + (n - 4).
\]  
By~\cref{claim:StructureFxFz}, \(|\mathcal{F}(x) \cup \mathcal{F}(z)| = 4 + 2(n - 4)\), it then follows that  
\[
|\mathcal{F}| \le \binom{n - 4}{2} + 4 + 2(n - 4) + (n - 4) = \binom{n - 1}{2} + 1,
\]  
contradicting our assumption. This completes the proof of~\cref{lemma: NoTwoB1}.
\end{proof}

Recall that 
\[
\mathcal{B} := \{ B_i : i \in [m],\ |B_i| = 2 \},
\] 
\[L=\big\{x\in [n]:\textup{there\ exists\ some\ }i\in [m]\ \textup{with\ }B_{i}=\{x\}\big\},\]
and
\[\mathcal{C}:=\{F_{i}\in\mathcal{F}:B_{i}=\emptyset\}.\]

By~\cref{lemma: NoTwoB1}, we then can assume that for any $x\in L$, \(\left| \{i \in [m] : B_{i} = \{x\} \} \right| \le 1.\)
Let us summarize the most important consequence of~\cref{lemma: NoTwoB1} here.

\begin{cor}\label{cor:MainInequality}
   Let $\mathcal{F}=\{F_{1},F_{2},\ldots,F_{m}\}\subseteq\binom{[n]}{3}$ be a set system with VC-dimension at most $2$. If $|\mathcal{F}|\ge \binom{n-1}{2}+2$, then we have 
   \[|\mathcal{F}|=|\mathcal{B}|+|L|+|\mathcal{C}|.\]
\end{cor}

\subsection{When $\mathcal{C}=\emptyset$}\label{subsection:CisEmpty}
When $\mathcal{C}=\emptyset$, we mainly take advantage of~\cref{cor:MainInequality} and focus on the auxiliary graph $\mathcal{B}$ and the set $L$. Our key claim is the following one.
\begin{claim}\label{claim:LxiaoyuN-2}
    If $\mathcal{C}=\emptyset$, then $|L|\le n-2$.
\end{claim}
\begin{poc}
Let \( \overline{\mathcal{B}} \) be the complement graph of \( \mathcal{B} \). By the definition of \( L \), for any \( x \in L \), there exists some set \( F_i = \{x, y, z\} \) such that \( B_i = \{x\} \). Thus, \( \{y, z\} \) serves as a transversal set of \( \mathcal{G}(x)=\{F\setminus\{x\}:F\in\mathcal{F}(x)\}. \) For convenience, we can regard $\mathcal{G}(x)$ as a graph on vertex set $[n]\setminus\{x\}$. By~\cref{claim:LargeDegree}, the vertices \( y \) and \( z \) share at least \( |\mathcal{F}(x)| - (n - 2)\ge 1 \) common neighbors in \( \mathcal{G}(x) \). Observe that for any such common neighbor $v \in [n] \setminus \{x, y, z\}$, we have $\{x, v\} \notin \mathcal{B}$. Additionally, it is clear that $\{x, y\}, \{x, z\} \notin \mathcal{B}$. Consequently, we obtain
\[
|N_{\overline{\mathcal{B}}}(x)| \geq 2 + |\mathcal{F}(x)| - (n - 2),
\]
which further implies that
\begin{align}\label{ineq-star}
|\mathcal{F}(x)| \leq n - 4 + |N_{\overline{\mathcal{B}}}(x)|.
\end{align}

If \( |\mathcal{B}| \leq \binom{n-1}{2} - n + 1 \), then by~\cref{cor:MainInequality}, we have  
\[
|\mathcal{F}| = |\mathcal{B}| + |L| \leq \binom{n-1}{2} + 1,
\]
which contradicts our assumption. Therefore, we proceed under the assumption that  
\[
|\mathcal{B}| \geq \binom{n-1}{2} - n + 2,
\]
which further implies  
\[
|\overline{\mathcal{B}}| \leq \binom{n}{2} - \binom{n-1}{2} + n - 2 = 2n - 3.
\]
Equivalently, we have  
\[
\sum_{a \in L} |N_{\overline{\mathcal{B}}}(a)| \leq 2|\overline{\mathcal{B}}|\leq 4n - 6.
\]
Suppose that $L=[n]$, then we have
\[\sum\limits_{a\in [n]}|\mathcal{F}(a)|\le n(n-4)+\sum\limits_{a\in L}|N_{\overline{\mathcal{B}}}(a)|\le n^{2}-6,\]
which further yields that when $n\ge 8$,
\[|\mathcal{F}|\le\frac{1}{3}\sum\limits_{a\in [n]}|\mathcal{F}(a)|\le\frac{n^{2}}{3}-2<\binom{n-1}{2}+1,\]
a contradiction. Therefore $|L|\le n-1$. We then claim that for any element \( b \in [n]\setminus L \), we have 
\begin{equation}\label{equ:Smalldegree}
    |N_{\overline{\mathcal{B}}}(b)| \leq n - 3.
\end{equation}
Suppose, for contradiction, that there exists some \( b \in [n]\setminus L  \) such that \( |N_{\overline{\mathcal{B}}}(b)| \geq n - 2 \). Let \( \mathcal{B}' \) denote the subgraph of \( \mathcal{B} \) induced by \( [n] \setminus \{b\} \). Observe that for each \( c \in [n] \setminus \{b\} \), we have \( |N_{{\mathcal{B}'}}(c)| \leq n - 2 \). By~\cref{claim:LowDegreeInB}, it follows that  
\[
|\mathcal{B}'| = \frac{1}{2} \sum_{c \in [n] \setminus \{b\}} |N_{{\mathcal{B}'}}(c)| \leq \frac{1}{2} \bigg( |L|\cdot (n - 4) + (n - 1 - |L|)\cdot(n - 2) \bigg) \leq \binom{n - 1}{2} - |L|.
\]

Thus, we have  
\[
|\mathcal{F}| = |\mathcal{B}| + |L| \leq |\mathcal{B}'| + 1 + |L| \leq \binom{n - 1}{2} + 1,
\]
which is a contradiction.  

Then suppose $|L|=n-1$, without loss of generality, we assume that $L=[n]\setminus\{1\}$, and by the above argument we have $|N_{\overline{\mathcal{B}}}(1)|\le n-3$ and hence  by \eqref{ineq-star},
\[|\mathcal{F}(1)|\le n-4+|N_{\overline{\mathcal{B}}}(1)|\le 2n-7.\]
Since $\sum\limits_{a\in L}|N_{\overline{\mathcal{B}}}(a)|\le 4n-6$ and $n\ge 8$, we then have
\[|\mathcal{F}|\le\frac{1}{3}\sum\limits_{a\in [n]}|\mathcal{F}(a)|\le \frac{1}{3}\bigg(2n-7+\sum\limits_{a\in L}(n-4+|N_{\overline{\mathcal{B}}}(a)|)\bigg)\le\frac{(n-1)(n-4)}{3}+2n-\frac{13}{3}\le\binom{n-1}{2}+1,\]
a contradiction to our assumption. Therefore $|L|\le n-2$, completing the proof.
\end{poc}
Based on~\cref{claim:LxiaoyuN-2}, we complete the proof by examining the maximum degree of the graph \( \mathcal{B} \), denoted by $\Delta(\mathcal{B})$. Obviously $\Delta(\mathcal{B})\le n-2$, since if $\{x,y\}\in\mathcal{B}$, then there exists some corresponding set $\{x,y,z\}\in\mathcal{F}$, which yields that $\{x,z\}\notin\mathcal{B}$. Moreover, when \( \Delta(\mathcal{B}) \leq n - 5 \), applying~\cref{cor:MainInequality}, we obtain
\[
|\mathcal{F}| = |\mathcal{B}| + |L| \leq \frac{n(n - 5)}{2} + n - 2 < \binom{n - 1}{2} + 1,
\]
which leads to a contradiction. Therefore, we proceed by dividing the final argument into three parts, each aimed at ruling out the case where \(\Delta(\mathcal{B}) = n - 2\), \(n - 3\), or \(n - 4\).

\subsubsection{$\Delta(\mathcal{B})=n-2$}
Let \( x \in [n] \) satisfy \( |N_{\mathcal{B}}(x)| = n - 2 \), and let \( y \notin N_{\mathcal{B}}(x) \). Then, for any \( v \in [n] \setminus \{x, y\} \), the set in \( \mathcal{F} \) corresponding to the witness \( \{x, v\} \in \mathcal{B} \) must be \( \{x, y, v\} \). This implies \( N_{\mathcal{B}}(y) = \emptyset \). 

We now claim that \( y \notin L \). Suppose, for contradiction, that \( y \in L \). Then, there exists some \( F_i = \{a, b, y\} \) such that \( B_i = \{y\} \). In this case, \( x \notin F_i \). Since \( n \geq 8 \), we can find \( w \in [n] \setminus \{x, y, a, b\} \) such that 
\[
\{x, y, w\} \cap \{a, b, y\} = \{y\}.
\]
However, this contradicts the assumption that \( B_i = \{y\} \). Thus, \( y \notin L \). Finally, by~\eqref{equ:Smalldegree}, we have \( |N_{\mathcal{B}}(y)| \geq 2 \), which is a contradiction to the fact \( N_{\mathcal{B}}(y) = \emptyset \). Therefore, it is impossible that $\Delta(\mathcal{B})=n-2$.

\subsubsection{$\Delta(\mathcal{B})=n-3$}
Let \( x \in [n] \) satisfy \( |N_{\mathcal{B}}(x)| = n - 3 \), and let \( y, z \notin N_{\mathcal{B}}(x) \). For each \( v \in [n] \setminus \{x, y, z\} \), the set in \( \mathcal{F} \) corresponding to \( \{x, v\} \in \mathcal{B} \) must contain either \( y \) or \( z \). Therefore, at least one of \( \{v, y\} \) and \( \{v, z\} \) does not belong to \( \mathcal{B} \), which further implies that
\[|N_{\mathcal{B}}(y)|+|N_{\mathcal{B}}(z)|\le 2(n-1)-(n-3)-2=n-1.\]
Moreover, we set $t:=\big|\big([n]\setminus\{y,z\}\big)\cap L\big|$. Note that $|L|-2\le t\le |L|$ by~\cref{claim:LowDegreeInB} and $\Delta(\mathcal{B})=n-3$, we then have
\[\sum\limits_{v\in [n]\setminus\{y,z\}}|N_{\mathcal{B}}(v)|=\sum\limits_{v\in [n]\setminus(\{y,z\}\cup L)}|N_{\mathcal{B}}(v)|+\sum\limits_{v\in ([n]\setminus\{y,z\})\cap L}|N_{\mathcal{B}}(v)|\le (n-2-t)\cdot (n-3)+t(n-4).\]
Therefore, we have
\[|\mathcal{B}|\le\frac{1}{2}\bigg(\big(n-2\big)\big(n-3\big)-t+n-1\bigg)\le\binom{n-2}{2}-\frac{|L|}{2}+\frac{n}{2}+\frac{1}{2}.\]
Furthermore, by~\cref{cor:MainInequality} and the fact that $|L|\le n-2$ from~\cref{claim:LxiaoyuN-2}, we have
\[|\mathcal{F}|=|\mathcal{B}|+|L|\le\binom{n-2}{2}+\frac{|L|}{2}+\frac{n}{2}+\frac{1}{2}\le\binom{n-1}{2}+\frac{3}{2},\]
which is a contradiction to the assumption that $|\mathcal{F}|\ge\binom{n-1}{2}+2$. Therefore, it is impossible that $\Delta(\mathcal{B})=n-3$.

\subsubsection{$\Delta(\mathcal{B})=n-4$}
Let \( x \in [n] \) satisfy \( |N_{\mathcal{B}}(x)| = n - 4 \), and let \( y, z, w \notin N_{\mathcal{B}}(x) \). For each \( v \in [n] \setminus \{x, y, z, w\} \), the set in \( \mathcal{F} \) corresponding to \( \{x, v\} \in \mathcal{B} \) must contain exactly one of \( y \), \( z \), or \( w \). Therefore, at least one of \( \{v, y\} \), \( \{v, z\} \), and \( \{v, w\} \) does not belong to \( \mathcal{B} \).  Moreover, $\{x,y\},\{x,z\},\{x,w\}\notin \mathcal{B}$. This observation implies that  
\begin{equation}\label{eq:Byzw}
   |N_{\mathcal{B}}(y)| + |N_{\mathcal{B}}(z)| + |N_{\mathcal{B}}(w)| \leq 3(n - 1) - (n - 4) - 3 = 2n - 2. 
\end{equation}
Since \( \Delta(\mathcal{B}) = n - 4 \), we have  
\begin{equation}\label{eq:BBB}
    |\mathcal{B}| \leq \frac{1}{2} \big( (n - 4)(n - 3) + 2n - 2 \big) = \binom{n - 3}{2} + n - 1.
\end{equation}
Applying~\cref{cor:MainInequality}, it follows that  
\begin{equation}\label{eq:+2}
    |\mathcal{F}| = |\mathcal{B}| + |L| \leq \binom{n - 3}{2} + n - 1 + n - 2 = \binom{n - 1}{2} + 2.
\end{equation}
Now, suppose the equality in~\eqref{eq:+2} holds, that is, \( |\mathcal{F}| = \binom{n - 1}{2} + 2 \). Then all of the equalities in~\cref{claim:LxiaoyuN-2},~\eqref{eq:Byzw} and~\eqref{eq:BBB} should hold. Therefore, we have \( |L| = n - 2 \), and for each \( v \in [n] \setminus \{y, z, w\} \), we have \( |N_{\mathcal{B}}(v)| = n - 4 \). Moreover, \( \{y, z\}, \{y, w\}, \{z, w\} \in \mathcal{B} \) since \(|N_{\mathcal{B}}(y)| + |N_{\mathcal{B}}(z)| + |N_{\mathcal{B}}(w)| = 2n - 2.\)  

We then claim that \( x \notin L \). By symmetry, showing this for \( x \) will imply the same for all \( v \in [n] \) with \( |N_{\mathcal{B}}(v)| = n - 4 \), leading to a contradiction since \( |L| = n - 2 \).  

Consider \( F_i = \{x, a, b\} \) with witness \( B_i = \{x\} \). Since $\{x,a\},\{x,b\}\notin \mathcal{B}$, we infer that $\{a,b\}\subseteq\{y,z,w\}$. However, it implies that $\{a,b\}\in \mathcal{B}$, contradicting our choice of $B_i=\{x\}$. Thus \( x \notin L \) and we cannot have equality in \eqref{eq:+2}.



Combining the above analysis with the inductive hypothesis, we conclude that when \(\mathcal{C} = \emptyset\), it holds that \(|\mathcal{F}| \le \binom{n - 1}{2} + 1\).

\subsection{Interplay among $\mathcal{B}$, $L$ and $\mathcal{C}$, and the forbidden structures in $\mathcal{C}$}\label{subsection:StructralofC}
Since we have already provided a complete proof for the case \( \mathcal{C} = \emptyset \) in~\cref{subsection:CisEmpty}, we will now proceed under the assumption that \( \mathcal{C} \neq \emptyset \).

First, for any \( F = \{x, y, z\} \in \mathcal{C} \), it is clear that \( \{x, y\}, \{y, z\}, \{x, z\} \notin \mathcal{B} \). To further analyze the intricate relationship among \( \mathcal{B} \), \( L \), and \( \mathcal{C} \), we will establish a series of claims. 

\begin{claim}\label{claim:MissTwoedges}
Let $F=\{x,y,z\}\in\mathcal{C}$, then for any element \( v \in [n] \setminus \{x, y, z\} \), at least two of the pairs in \(\big\{ \{v, x\}, \{v, y\}, \{v, z\} \big\}\) do not belong to \( \mathcal{B} \).

\end{claim}
\begin{poc}
Suppose that there are two pairs in \(\big\{ \{v, x\}, \{v, y\}, \{v, z\} \big\}\) belonging to $\mathcal{B}$.
By symmetry, we assume that that \( F_{i} = \{x, v, a\} \) with witness \( B_{i} = \{x, v\} \), and \( F_{j} = \{y, v, b\} \) with witness \( B_{j} = \{y, v\} \). Then all sets in \( \mathcal{F}(v) \setminus \{F_{i}, F_{j}\} \) must contain \( z \). If \( z \notin \{a, b\} \), then  
\[
|\mathcal{F}(v)| \le 2 + (n - 4) = n - 2,
\]  
contradicting~\cref{claim:LargeDegree}. If exactly one of \( a \) or \( b \) equals \( z \), then  
\[
|\mathcal{F}(v)| \le 1 + (n - 3) = n - 2,
\]  
again a contradiction. If \( a = b = z \), then  
\[
|\mathcal{F}(v)| \le n - 2,
\]  
which is still a contradiction. This finishes the proof.
\end{poc}

We next show that under the assumption that \(|\mathcal{F}| \ge \binom{n-1}{2} + 2\), for any set \(F \in \mathcal{C}\), it is impossible that all elements of \(F\) belong to \(L\). For convenience, given subsets \(A \subseteq B\) and set system $\mathcal{F}$, we define  
\[
\mathcal{F}(A,B) = \{F \setminus B : F \in \mathcal{F}, \ F \cap B = A\}.
\]

\begin{claim}\label{claim:XYZLE2}
For any $F\in\mathcal{C}$, we have
 \(|L\cap F|\le 2 \).
\end{claim}
\begin{poc}
Suppose that $\{x,y,z\}\in \mathcal{C}$ and \(\{x, y, z\} \subseteq L \). There exist \( F_1, F_2, F_3 \in \mathcal{F} \) such that  
\[
{B_1 = \{x\}, \quad B_2 = \{y\}, \quad B_3 = \{z\}.}
\]

Since \( \{x, y, z\} \in \mathcal{F} \), we have \( |F_i \cap \{x, y, z\}| \geq 2 \) for each \( i \in [3] \). Therefore, we can write  
\[
F_1 = \{x, u_x, v_x\}, \quad F_2 = \{y, u_y, v_y\}, \quad F_3 = \{z, u_z, v_z\},
\]
where \( u_x, u_y, u_z \in \{x, y, z\} \). 


Note that $F_1\cup F_2\cup F_3=\{x,y,z,v_x,v_y,v_z\}$. Since $n\ge 8$, pick arbitrary $w\in [n]\setminus \{x,y,z,v_x,v_y,v_z\}$.
Since $B_1=\{x\}$, we infer $|F\cap F_1|= 2$ for each $F\in \cF(w,x)$. It follows that $|\cF(w,x)|\leq 2$. Similarly $|\cF(w,y)|\leq 2$ and $|\cF(w,z)|\leq 2$. 
Since $\{x,y,z\}$ is a transversal of $\cF$, we obtain that 
\[
|\cF(w)|=|\cF(w,x)|+|\cF(w,y)|+|\cF(w,z)| \leq 2+2+2=6<n-1,
\]
contradicting~\cref{claim:LargeDegree}. This finishes the proof.
\end{poc}

Then we can see, if $\mathcal{C}\neq\emptyset$, based on~\cref{claim:LowDegreeInB},~\cref{cor:MainInequality},~\cref{claim:MissTwoedges}, and~\cref{claim:XYZLE2}, we have
\[
|\mathcal{F}| \le \binom{n}{2} - 3 - 2(n - 3) - \frac{|L| - 2}{2} + |L| + |\mathcal{C}|,
\]
where the term \( 2(n - 3)  \) comes from~\cref{claim:MissTwoedges}. Moreover, the term $\frac{|L|-2}{2}$ comes from~\cref{claim:LowDegreeInB} and~\cref{claim:XYZLE2}, more precisely, for each element \( v \in L \setminus \{x, y, z\} \), there are at least three elements \( u \in [n] \setminus \{v\} \) such that \( \{u, v\} \notin \mathcal{B} \), and some two of these \( u \) belong to \( \{x, y, z\} \), therefore, there are at least $\frac{|L|-2}{2}$ extra pairs of elements $\{u,v\}$ not appearing in $\mathcal{B}$. Furthermore, by~\cref{claim:XYZLE2}, we can see if $\mathcal{C}\neq\emptyset$, $|L|\le n-1$, therefore we have
\begin{align*}
|\mathcal{F}| &\le \binom{n}{2} - 3 - 2(n - 3) - \frac{|L| - 2}{2} + |L| + |\mathcal{C}|\\[3pt]
&\le \binom{n-1}{2}+1 - n+2+ \frac{|L| }{2} + |\mathcal{C}|\\[3pt]
&\le \binom{n-1}{2}+1 - \frac{n}{2}+2 - \frac{1 }{2} + |\mathcal{C}|\\[3pt]
&\leq  \binom{n - 1}{2} + 1+|\mathcal{C}|-2-\frac{1}{2},
\end{align*}
where the last inequality holds since $\frac{n}{2}\ge 4$. Since $|\mathcal{F}|$ must be an integer, from now on, we can always assume that $|\mathcal{C}|\ge 4$, otherwise $|\mathcal{F}|<\binom{n-1}{2}+2$, a contradiction.

Next, we further explore the structural properties of \( \mathcal{C} \) under the assumption that \( |\mathcal{C}| \geq 4 \). Our primary focus here is to identify and eliminate certain forbidden configurations within \( \mathcal{C} \). First, we say a set system \( \mathcal{F} \) is \emph{$2$-intersecting} if for any two sets \( A, B \in \mathcal{F} \), we have \(|A \cap B| \geq 2.\) We also say two set systems \( \mathcal{F}, \mathcal{F}' \) are \emph{cross-intersecting} if $F\cap F'\neq \emptyset$ for any two sets $F\in \mathcal{F}$, $F'\in \mathcal{F}'$.

\begin{claim}\label{claim:Not2Intersection}
If $|\mathcal{C}|\ge 4$, then $\mathcal{C}$ is not 2-intersecting.
\end{claim}
\begin{poc}
    Suppose that \( \mathcal{C} \) is \( 2 \)-intersecting and let \( \{x, y, z\} \in \mathcal{C} \). Observe that 
 \( \mathcal{C}(\{x, y\}, \{x, y, z\}) \), \( \mathcal{C}(\{x, z\}, \{x, y, z\}) \), \( \mathcal{C}(\{y, z\}, \{x, y, z\}) \) are pairwise cross-intersecting. 
Suppose \(|\mathcal{C}(\{x, y\}, \{x, y, z\})|\geq 2 \), then since $\mathcal{C}$ is $2$-intersecting, we can see \( \mathcal{C}(\{x, z\}, \{x, y, z\}) =\mathcal{C}(\{y, z\}, \{x, y, z\})=\emptyset \). Moreover, we claim that \(|\mathcal{C}(\{x, y\}, \{x, y, z\})|\leq 2 \), otherwise, there exist sets $\{x,y,z_i\}\in \mathcal{C}$, for $i=1,2,3$. Then by the maximality of witness, for $\{x,y,z\}\in\mathcal{C}$, there is some $F\in\mathcal{F}$ such that $F\cap\{x,y,z\}=\{z\}$, however, $F$ must contain $z_{1},z_{2},z_{3}$, which is impossible. But it this case, $|\mathcal{C}|=3$, a contradiction.

Therefore, all of \( |\mathcal{C}(\{x, y\}, \{x, y, z\})| \), \( |\mathcal{C}(\{x, z\}, \{x, y, z\})| \), \(| \mathcal{C}(\{y, z\}, \{x, y, z\})|\) are at most $1$. Since $|\mathcal{C}|\ge 4$ and $\mathcal{C}$ is $2$-intersecting, \( \mathcal{C} \) has to be a complete \( 3 \)-uniform hypergraph on \( \{x, y, z, w\} \) for some $w\in [n]\setminus\{x,y,z\}$. 
Let $v\in [n]\setminus \{x,y,z,w\}$. Then for each $F\in \mathcal{F}(v)$, $|F\cap \{x,y,z,w\}|=2$. It follows that $|\mathcal{F}(v)|\leq \binom{4}{2}=6$, contradicting Claim \ref{claim:LargeDegree}. 
\end{poc}

We refer to a 3-uniform hypergraph on the vertex set \(\{x, y, z, a, b, c\}\) with edges \(\{a, x, b\}\), \(\{b, z, c\}\), and \(\{a, y, c\}\) as a \emph{linear triangle}, a configuration that has been extensively studied in extremal combinatorics.
\begin{claim}\label{claim:NoLinearTriangle}
    $\mathcal{C}$ contains no linear triangle.
\end{claim}
\begin{poc}
   Suppose there is a set \( W = \{a, b, c, x, y, z\} \subseteq [n] \) such that \( \{a, x, b\}, \{a, y, c\}, \{b, z, c\} \in \mathcal{C} \). Since \( n \geq 8 \), there exists some vertex \( v \in [n] \setminus W \). By~\cref{claim:LargeDegree}, we have  
\[
|\mathcal{F}(v)| \geq n - 1 \geq 7.
\]
Furthermore, every set in \( \mathcal{G}(v) \) must intersect each of \( \{a, x, b\} \), \( \{a, y, c\} \), and \( \{b, z, c\} \). However, the transversal number of a linear triangle is \( 2 \), and there are exactly \( 6 \) minimal transversal sets for this linear triangle:  
\[
\{a, b\}, \quad \{a, c\}, \quad \{b, c\}, \quad \{a, z\}, \quad \{b, y\}, \quad \{c, x\}.
\]
This contradicts the fact that \( |\mathcal{F}(v)| \geq 7 \). This finishes the proof.
\end{poc}

\subsection{Transversal number of $\mathcal{C}$}\label{subsection:Transversalof C}
Building on the structural properties established in~\cref{subsection:StructralofC}, we proceed by partitioning our analysis according to the transversal number of \(\mathcal{C}\), denoted by \(\tau(\mathcal{C})\), which is defined as the minimum size of a set that intersects every set in \(\mathcal{C}\). Since $\mathcal{C}$ is intersecting, clearly we have $\tau(\mathcal{C})\le 3$.

We first simply show that the transversal number of $\mathcal{C}$ cannot be three.

\begin{claim}\label{claim:TauLe22 }
    \(\tau(\mathcal{C})\le 2.\)
\end{claim}
\begin{poc}
 Suppose that $\tau(\mathcal{C})=3$. Let \( \{x, y, z\} \in \mathcal{C} \). We claim that  
\[
|\mathcal{F}(\{x, y\}, \{x, y, z\})| \leq 2.
\]
To see this, note that \( \{x, y\} \) cannot be a transversal set of \( \mathcal{C} \) since \( \tau(\mathcal{C}) = 3 \). Therefore, there exists some \( F \in \mathcal{C} \) such that \( F \cap \{x, y\} = \emptyset \) and \( z \in F \).  Let \( F = \{z, a, b\} \). Then, any set in \( \mathcal{F}(\{x, y\}, \{x, y, z\}) \) must contain either \( a \) or \( b \). This implies \(|\mathcal{F}(\{x, y\}, \{x, y, z\})| \leq 2.\)

 We further claim that \( |\mathcal{F}(\{x\}, \{x, y, z\})| \leq 4 \). Consider any set \( E = \{u, v\} \in \mathcal{F}(\{x\}, \{x, y, z\}) \). This set can be generated as follows:  

\begin{enumerate}
    \item Since \( \tau(\mathcal{C}) = 3 \), there exists some \( F_{1} \in \mathcal{C} \) such that \( x \notin F_{1} \). Therefore, we must have \( u \in F_{1} \setminus \{y, z\} \).
    \item Applying \( \tau(\mathcal{C}) = 3 \) again, there exists some \( F_{2} \in \mathcal{C} \) such that \( x, u \notin F_{2} \). Thus, \( v \in F_{2} \setminus \{x, y, z\} \). 
\end{enumerate}  
Since both \( F_{1} \) and \( F_{2} \) must intersect \( \{y, z\} \), the number of valid choices for \( u \) and \( v \) is at most \( 2 \) each. Therefore,  
\[
|\mathcal{F}(\{x\}, \{x, y, z\})| \leq 2 \times 2 = 4.
\]
By symmetry, applying the same argument to \( \{y\} \) and \( \{z\} \), we have  
\[
|\mathcal{F}| \leq 3 \times 2 + 3 \times 4 + 1 = 19 < \binom{n - 1}{2} + 1,
\]
where the last inequality holds since \( n \geq 8 \). This completes the proof.
\end{poc}
We then proceed by dividing the final argument into two parts, each aimed at ruling out the cases  $\tau(\mathcal{C})=2$ and $\tau(\mathcal{C})=1$.

\subsubsection{$\tau(\mathcal{C})=2$}
When $\tau(\mathcal{F})=2$, we can first provide a lower bound for $|\mathcal{C}|$.
\begin{claim}\label{claim:Tau2C7}
    If $\tau(\mathcal{C})=2$, then $|\mathcal{C}|\ge 7$.
\end{claim}
\begin{poc}
By \cref{claim:XYZLE2}, if $|L| = n-1$, then $\tau(\mathcal{C}) = 1$; hence, we must have $|L| \le n-2$. By~\cref{claim:Not2Intersection}, there exist two sets in \(\mathcal{C}\) whose intersection has size one. Let \(\{x, y, z\}, \{z, u, v\} \in \mathcal{C}\) be such sets with \(\{x, y, z\} \cap \{z, u, v\} = \{z\}\). Since \(\tau(\mathcal{C}) = 2\), the element \(z\) alone cannot intersect all sets in \(\mathcal{C}\). Therefore, there exists some \(F \in \mathcal{C}\) such that \(F \cap \{x, y\} \neq \emptyset\) and \(F \cap \{u, v\} \neq \emptyset\).  

By \cref{claim:NoLinearTriangle}, we have $F \subseteq \{x, y, u, v\}$. Without loss of generality, assume $F = \{x, u, v\}$. By \cref{claim:MissTwoedges}, for each $w \in [n] \setminus \{x, y, z, u, v\}$, at least three of the pairs $\{w, x\}, \{w, y\}, \{w, z\}, \{w, u\}, \{w, v\}$ are not in $\mathcal{B}$. Additionally,  
\begin{itemize}
\item At least one of $\{y, u\}$ or $\{y, v\}$ is not in $\mathcal{B}$.  
\item The pairs $\{x, y\}, \{x, z\}, \{y, z\}, \{x, u\}, \{x, v\}, \{u, v\}, \{z, u\}, \{z, v\}$ are not in $\mathcal{B}$.
\end{itemize}
In total, the number of missing pairs in $\mathcal{B}$ is at least $3(n - 5) + 1 + 8 = 3n - 6$. Then by~\cref{lemma: NoTwoB1},
\begin{align*}
|\mathcal{F}| &\le \binom{n}{2} -  (3n - 6)  + |L| + |\mathcal{C}|\\[3pt]
&\le \binom{n-1}{2}-2n+5+ n-2 + |\mathcal{C}|\\[3pt]
&= \binom{n-1}{2}+1 -n+2 + |\mathcal{C}|.
\end{align*}
Since $n\ge 8$, if $|\mathcal{C}|\leq 6$, then $|\mathcal{F}|\le\binom{n-1}{2}+1$, a contradiction to our assumption. This finishes the proof.
\end{poc}
Let \(\{x, y\}\)\footnote{Here the set \{x,y\} is unrelated to the one in~\cref{claim:Tau2C7}} be a transversal of \(\mathcal{C}\). Since neither \(x\) nor \(y\) alone intersects all sets in \(\mathcal{C}\), it follows that both subfamilies \(\mathcal{C}(\{x\}, \{x, y\})\) and \(\mathcal{C}(\{y\}, \{x, y\})\) are non-empty. Furthermore, these two families are cross-intersecting. 

We now show that each of $\mathcal{C}(\{x\}, \{x, y\})$ and $\mathcal{C}(\{y\}, \{x, y\})$ is intersecting. By symmetry, it suffices to consider $\mathcal{C}(\{x\}, \{x, y\})$. Suppose $\{a, b\}, \{c, d\} \in \mathcal{C}(\{x\}, \{x, y\})$. Then, any set in $\mathcal{C}(\{y\}, \{x, y\})$ must intersect both $\{a, b\}$ and $\{c, d\}$. Without loss of generality, assume $\{y, a, c\} \in \mathcal{C}$. Then, the sets $\{x, a, b\}$, $\{x, c, d\}$, and $\{y, a, c\}$ form a linear triangle in $\mathcal{C}$, contradicting \cref{claim:NoLinearTriangle}. Moreover, we have
\begin{equation}\label{trans2Codegree}
    |\mathcal{C}(x,y)|\le 3.
\end{equation}
To see this, suppose there are sets $\{x,y,z_{i}\}\in\mathcal{C}$ for $i=1,2,3,4$. then by maximality, there exists some set $F\in\mathcal{F}$ such that $F\cap\{x,y,z_{1}\}=\{z_{1}\}$, however, $F$ must contain $z_{2},z_{3},z_{4}$, which is impossible.

Now we regard $\mathcal{C}(\{x\}, \{x, y\})\cup \mathcal{C}(\{y\}, \{x, y\})$ as a graph. By the above analysis, we can see it is intersecting, therefore, it is either a star, or a triangle.

If $\mathcal{C}(\{x\}, \{x, y\})\cup \mathcal{C}(\{y\}, \{x, y\})$ is a star with center $z$. By~\cref{claim:Tau2C7} and~\eqref{trans2Codegree}, we can see the number of edges in the graph $\mathcal{C}(\{x\}, \{x, y\})\cup \mathcal{C}(\{y\}, \{x, y\})|$ is at least $2$, which yields that $|F\cap \{x,y,z\}|\geq 2$ for each $F\in \mathcal{C}$.

A crucial observation is that, suppose $\{x,y,z\}\in \mathcal{C}$, then by~\cref{claim:Tau2C7} and~\eqref{trans2Codegree}, we can see 
$|\mathcal{C}(\{x,y\},\{x,y,z\})|=|\mathcal{C}(\{x,z\},\{x,y,z\})|=|\mathcal{C}(\{y,z\},\{x,y,z\})|=2$, and $|\mathcal{C}|=7$. By~\cref{claim:NoLinearTriangle}, $\mathcal{C}$ does not contain a linear triangle, therefore, there exist distinct elements $u,v\in [n]\setminus\{x,y,z\}$ such that 
\[
\mathcal{C} =\{\{x,y,z\},\{x,y,u\},\{x,y,v\},\{x,z,u\},\{x,z,v\},\{y,z,u\},\{y,z,v\} \}.
\]
Then in this case, the all possible minimal transversal sets of size 2 in $\mathcal{C}$ are $\{\{x,y\},\{x,z\},\{y,z\}\}$. Since $n\ge 8$, pick $w\in [n]\setminus \{x,y,z,u,v\}$. Then we have $|\mathcal{F}(w)|\leq 3<n-1$, contradicting~\cref{claim:LargeDegree}. 

Therefore we can assume $\{x,y,z\}\notin \mathcal{C}$. In this case, note that 
\[
|\mathcal{C}(\{x,y\},\{x,y,z\})|+|\mathcal{C}(\{x,z\},\{x,y,z\})|+|\mathcal{C}(\{y,z\},\{x,y,z\})|=|\mathcal{C}|\geq 7.
\]
 By symmetry assume 
\[
|\mathcal{C}(\{x,y\},\{x,y,z\})|\leq |\mathcal{C}(\{x,z\},\{x,y,z\})|\leq |\mathcal{C}(\{y,z\},\{x,y,z\})|. 
\]
By~\eqref{trans2Codegree}, $|\mathcal{C}(\{x,y\},\{x,y,z\})|\leq 3$, $|\mathcal{C}(\{x,z\},\{x,y,z\})|\leq 3$, $|\mathcal{C}(\{y,z\},\{x,y,z\})|\leq 3$, we have
\[
|\mathcal{C}(\{x,y\},\{x,y,z\})|\geq 1,\ |\mathcal{C}(\{x,z\},\{x,y,z\})|\geq 2,\ |\mathcal{C}(\{y,z\},\{x,y,z\})|\geq 3.
\]
Thus  one can greedily find a linear triangle, a contradiction to~\cref{claim:NoLinearTriangle}. Therefore, it is impossible that $\mathcal{C}(\{x\}, \{x, y\})\cup \mathcal{C}(\{y\}, \{x, y\})$ is a star. 

It remains to show the case that $\mathcal{C}(\{x\}, \{x, y\})\cup \mathcal{C}(\{y\}, \{x, y\})$ is a triangle on $\{a,b,c\}$ cannot occur. Note that $|\mathcal{C}|\geq 7$ by~\cref{claim:Tau2C7} and $|\mathcal{C}(x,y)|\leq 3$ by~\eqref{trans2Codegree}, we have
\[
|\mathcal{C}(\{x\}, \{x, y\})|+|\mathcal{C}(\{y\}, \{x, y\})|\geq 4.
\]
By pigeonhole principle and by symmetry, we assume that $\{x,a,b\},\{y,a,c\}\in \mathcal{C}$. By~\cref{claim:NoLinearTriangle}, $\mathcal{C}$ does not contain a linear triangle, we can see that $\mathcal{C}\subseteq \binom{\{x,y,a,b,c\}}{3}$. Observe that for each triple $S\in \mathcal{C}$, $\{x,y,a,b,c\}\setminus S$ cannot be a transversal set of $\mathcal{C}$. Since $|\mathcal{C}(\{x\}, \{x, y\})|+|\mathcal{C}(\{y\}, \{x, y\})|\geq 4$, and each set in $\mathcal{C}$ must intersect $\{x,y\}$, we infer that the number of minimal transversal sets of size $2$ in $\mathcal{C}$ is at most $\binom{5}{2}-4=6$. Moreover, since $n\ge 8$, pick $p\in [n]\setminus\{x,y,a,b,c\}$. We can see that $|\mathcal{F}(p)|\leq 6<n-1$, contradicting~\cref{claim:LargeDegree}. 

Based on the above analysis, it is impossible that $\tau(\mathcal{C})=2$.

\subsubsection{$\tau(\mathcal{C})=1$}
$\tau(\mathcal{C})$ means that $\mathcal{C}$ is a star, by~\cref{claim:Not2Intersection}, $\mathcal{C}$ is not $2$-intersecting, and since we assume $|\mathcal{C}|\ge 4$, there exist two sets in \(\mathcal{C}\) whose intersection has size one. Let \(\{x, y, z\}, \{z, u, v\} \in \mathcal{C}\) be such sets with \(\{x, y, z\} \cap \{z, u, v\} = \{z\}\). Indeed, $z$ is also the center of the star $\mathcal{C}$. Then by~\cref{claim:MissTwoedges}, for each $w\in [n]\setminus\{x,y,z,u,v\}$, at least $3$ of $\{x,w\}$, $\{y,w\}$, $\{z,w\}$, $\{u,w\}$ and $\{v,w\}$ do not belong to $\mathcal{B}$. Moreover, by~\cref{claim:MissTwoedges} again, at least one of $\{y,u\}$ and $\{y,v\}$ is not in $\mathcal{B}$ and at least one of $\{x,u\}$ and $\{x,v\}$ is not in $\mathcal{B}$. It is also obvious that 
\[\big\{\{x,y\},\{x,z\},\{y,z\},\{u,v\},\{u,z\},\{v,z\}\big\}\cap\mathcal{B}=\emptyset.\]
Therefore, there are at least $3(n-5)+6+2$ pairs are not in $\mathcal{B}$. Then by~\cref{lemma: NoTwoB1}, we have
\begin{align*}
|\mathcal{F}| &\le \binom{n}{2} -  (3n - 7)  + |L| + |\mathcal{C}|\\[3pt]
&\le \binom{n-1}{2}-2n+6+ n-1 + |\mathcal{C}|\\[3pt]
&\le \binom{n-1}{2}+1 -n+4 + |\mathcal{C}|\\[3pt]
&\leq  \binom{n - 1}{2} + 1+|\mathcal{C}|-4.
\end{align*}
Therefore we can assume that \(|\mathcal{C}| \ge 5\), otherwise $|\mathcal{F}|\le\binom{n-1}{2}+1$, a contradiction to our assumption. We then define a 2-uniform set system \(\mathcal{Z}\) on the ground set \(Z\subseteq [n] \setminus \{z\}\) as  
\[
\mathcal{Z} := \{F \setminus \{z\} : F \in \mathcal{C} \subseteq \mathcal{F} \}.
\]  
Equivalently, we may view \(\mathcal{Z}\) as a graph on the vertex set \(Z\subseteq [n] \setminus \{z\}\), where $Z$ consists of all elements which appear in some set of $\mathcal{Z}$. Note that \(|\mathcal{Z}| = |\mathcal{C}| \ge 5\), our next goal is to establish several useful properties of the graph \(\mathcal{Z}\).
\begin{claim}\label{claim:PropertyZ}
The followings hold.
\begin{enumerate}
    \item[\textup{(1)}] $\mathcal{Z}$ has matching number  $2$ and maximum degree at most $3$.
    \item[\textup{(2)}] $\mathcal{Z}$ is $K_{2,2}$-free.
     \item[\textup{(3)}] If there is a vertex $u$ having $3$ neighbors in graph $\mathcal{Z}$, say $\{v_{1},v_{2},v_{3}\}$, then each other edge in $\mathcal{Z}$ must contain at least one vertex in $\{u,v_{1},v_{2},v_{3}\}.$
     \item[\textup{(4)}] If $|Z|\le 6$, then the transversal number of $\mathcal{Z}$ is $2$, and the number of $2$-element transversal sets of $\mathcal{Z}$ is at least $2$.
\end{enumerate}
\end{claim}
\begin{poc}
    For (1), suppose that there is a vertex $a$ having four neighbors $b_{1},b_{2},b_{3},b_{4}$ in $\mathcal{Z}$. Note that by maximality of witness, there exists a set $F\in\mathcal{F}$ such that $F\cap\{z,a,b_{1}\}=\{b_{1}\}$, and $F$ must contain $b_{2},b_{3},b_{4}$, which is impossible.
    
    Suppose that there is a matching $\{a_{1},b_{1}\},\{a_{2},b_{2}\}$ and $\{a_{3},b_{3}\}$ in $\mathcal{Z}$, then by maximality of witness, there exists a set $F\in \mathcal{F}$ such that $F\cap \{z,a_{1},b_{1}\}=\{a_{1},b_{1}\}$, and $F$ must intersect $\{a_{2},b_{2}\}$ and $\{a_{3},b_{3}\}$, which is impossible. Since $\mathcal{Z}$ has at least $5$ edges, and the maximum degree of $\mathcal{Z}$ is at most $3$, it is easy to see $\mathcal{Z}$ contains a matching of size $2$.

    For (2), Suppose that $a,b,c,d$ form a copy of $K_{2,2}$ in $\mathcal{Z}$, then $\{z,a,c\},\{z,a,d\},\{z,b,c\},\{z,b,d\}\in \mathcal{C}$. Consider the set $\{a,d,z\}$. By maximality of witness, there exist sets $F_{1},F_{2},F_{3}\in\mathcal{F}$ with the following properties.
    \begin{enumerate}
        \item $F_{1}\cap\{a,d,z\}=\{a,d\}$. Since $\{z,b,c\}\in\mathcal{C}$, we can see $F_{1}$ is either $\{a,b,d\}$ or $\{a,c,d\}$.
        \item $F_{2}\cap\{a,d,z\}=\{a\}$, since $\{z,b,d\}\in\mathcal{C}$, $F_{2}$ contains $b$. 
        \item $F_{3}\cap\{a,d,z\}=\{d\}$, since $\{z,a,c\}\in\mathcal{C}$, $F_{3}$ contains $c$.
    \end{enumerate}
Since $z\notin F_{1}$, $F_{1}\notin \mathcal{C}$, we conclude that the size of witness of $F_{1}$ is either $1$ or $2$. We now examine the possible forms of $F_{1}$ accordingly.

If \( F_1 = \{a, c, d\} \), we can immediately observe the following intersections: \begin{enumerate}
    \item \( \{z, b, d\} \cap \{a, c, d\} = \{d\} \) and \( \{z, b, c\} \cap \{a, c, d\} = \{c\} \). 
    \item \( \{z, a, c\} \cap \{a, c, d\} = \{a, c\} \), \( \{z, a, d\} \cap \{a, c, d\} = \{a, d\} \) and \( F_3 \cap \{a, c, d\} = \{c, d\} \).  
\end{enumerate}  
Furthermore, if there is no set \( F \in \mathcal{F} \) such that \( F \cap \{a, c, d\} = \{a\} \), then the family \( \mathcal{F}(a) \) can only consist of the following sets:  
\[
\{a, c, d\}, \{a, b, c\}, \{z, a, c\}, \{z, a, d\}, \{a, b, d\}.
\]  
This leads to a contradiction, as \( |\mathcal{F}(a)| \geq n - 1 \geq 7 \) by \cref{claim:LargeDegree}. Therefore, if $F_{1}=\{a,c,d\}$, we cannot find a witness set of size $1$ or $2$, a contradiction. 

If $F_{1}=\{a,b,d\}$, we have the following intersections:
\begin{enumerate}
    \item $\{z,a,c\}\cap\{a,b,d\}=\{a\}$, $\{z,b,c\}\cap\{a,b,d\}=\{b\}$.
    \item $\{z,a,d\}\cap\{a,b,d\}=\{a,d\}$, $F_{2}\cap\{a,b,d\}=\{a,b\}$ and $\{z,b,d\}\cap\{a,b,d\}=\{b,d\}$.
\end{enumerate}
Furthermore, if there is no set $F\in\mathcal{F}$ such that $F\cap\{a,b,d\}=\{d\}$, then $\mathcal{F}(d)$ can only consist of the following sets:
\[\{z,a,d\},\{a,c,d\},\{a,b,d\},\{b,c,d\},\{z,b,d\},\]
which also contradicts~\cref{claim:LargeDegree}. Therefore, $F_{1}$ cannot be $\{a,b,d\}$. This completes the proof of property (2).

For (3), suppose there exists an edge \(\{x, y\} \in \mathcal{Z}\) such that \(\{x, y\} \cap \{u, v_1, v_2, v_3\} = \emptyset\). Consider the set \(\{u, v_1, z\} \in \mathcal{C}\). By the maximality of witness, there exists a set \(F \in \mathcal{F}\) such that \(F \cap \{u, v_1, z\} = \{v\}\). This implies that \(F\) must contain \(\{v_2, v_3\}\) and satisfy \(F \cap \{x, y\} \neq \emptyset\). However, this leads to a contradiction, as \(F\) cannot simultaneously contain \(\{v_2, v_3\}\) and intersect with \(\{x, y\}\). This completes the proof of (3).

For (4), first since the matching number of $\mathcal{Z}$ is $2$, the transversal number of $\mathcal{Z}$ is at least $2$. If \( |Z| \leq 6 \), suppose that the transversal number of $\mathcal{Z}$ is $3$, since $n\ge 8$, for any element $v\in [n]\setminus\{Z\cup \{z\}\}$, $\mathcal{F}(v)$ must contain $z$, which implies $|\mathcal{F}(v)|\le n-2$, a contradiction to to~\cref{claim:LargeDegree}. Therefore we conclude the transversal number of $\mathcal{Z}$ is $2$.

Now suppose that \( \{a, b\} \) is the unique \( 2 \)-element transversal set. Then if $x\in [n]\setminus\{Z\cup\{z\}\}$, since \( |\mathcal{F}(x)| \geq n-1 \) by~\cref{claim:LargeDegree}, the set \( \mathcal{F}(x) \) must take the form:
\[
\mathcal{F}(x) = \{x, a, b\} \cup \left( \bigcup_{v \in [n] \setminus \{x, z\}} \{x, z, v\} \right).
\]
Furthermore, since \( \{a, b\} \) is the minimal transversal set, there exist edges \( \{a, c_1\} \) and \( \{b, c_2\} \) in \( \mathcal{Z} \). Since $\{x,a,b\}\notin\mathcal{C}$, the witness of $\{x,a,b\}$ must be of size $1$ or $2$. However, observe that
\begin{enumerate}
    \item $\{x,z,a\}\cap\{x,a,b\}=\{x,a\}$, $\{x,z,b\}\cap\{x,a,b\}=\{x,b\}$ and $\{z,a,b\}\cap\{x,a,b\}=\{a,b\}$.
    \item $\{z,a,c_{1}\}\cap\{x,a,b\}=\{a\}$, $\{z,b,c_{2}\}\cap\{x,a,b\}=\{b\}$ and $\{x,z,v\}\cap\{x,a,b\}=\{x\}$ for any $v\notin \{a,b,x\}$. 
\end{enumerate}
Therefore, we cannot find a non-empty witness of $\{x,a,b\}$, a contradiction. This finishes the proof.
\end{poc}
Based on \cref{claim:PropertyZ}(1), let \( \{a_1, a_2\} \) and \( \{b_1, b_2\} \) form a matching of size \( 2 \) in \( \mathcal{Z} \). Suppose \( Z \) contains at least \( 7 \) elements, then let \( \{a, b, c\} \subseteq Z \setminus \{a_1, a_2, b_1, b_2\} \). Since the matching number of \( \mathcal{Z} \) is \(2\) and the maximum degree of \( \mathcal{Z} \) is at most \( 3 \), then there exists some set $T=\{a_{i},b_{j}\}$ for \( 1 \leq i, j \leq 2 \), such that each element in \( \{a, b, c\} \) must be adjacent to at least one element in \( T  \). By pigeonhole principle and by symmetry, we can assume that $i=1$ and $a_{1}$ is adjacent to both $a$ and $b$, then the degree of $a_{1}$ in $\mathcal{Z}$ is $3$. However, the set $\{b_{1},b_{2}\}$ is disjoint with $\{a_{1},a_{2},a,b\}$, which is a contradiction with \cref{claim:PropertyZ}(3).

Therefore, we can assume that \( Z \) contains at most \( 6 \) elements. Since \( \mathcal{Z} \) is \( K_{2,2} \)-free by \cref{claim:PropertyZ}(2) and the number of edges in \( \mathcal{Z} \) is at least \( 5 \), without loss of generality, we may assume there exists an element \( x \in Z\setminus\{a_{1},a_{2},b_{1},b_{2}\} \) such that \( \{x, a_1\} \) is an edge in \( \mathcal{Z} \). Let \( \mathcal{Z}' \) be the subgraph of \( \mathcal{Z} \) induced by the vertices \( \{x, a_1, a_2, b_1, b_2\} \) with edges \( \{a_1, a_2\} \), \( \{b_1, b_2\} \), and \( \{a_1, x\} \). Note that \( \{a_1, b_1\} \) and \( \{a_1, b_2\} \) are the only \( 2 \)-element transversal sets of \( \mathcal{Z}' \), and any \( 2 \)-element transversal set of \( \mathcal{Z} \) must also be a \( 2 \)-element transversal set of \( \mathcal{Z}' \). 

Now, observe that each element in \( \{a_2, x, b_1, b_2\} \) cannot have a neighbor other than \( a_1 \); otherwise, the number of \( 2 \)-element transversal sets would be less than \( 2 \), contradicting \cref{claim:PropertyZ}(4). However, since \( a_1 \) has at most \( 3 \) neighbors, this implies that \( \mathcal{Z} \) has at most \( 4 \) edges, which contradicts our initial assumption that $|\mathcal{Z}|\ge 5$. This completes the proof.

\section{Concluding remarks and future work}
In this paper, we investigate a fundamental open problem in extremal set theory: determining the maximum size of a \((d+1)\)-uniform set system on \([n]\) with VC-dimension at most \(d\), for any given \(d \ge 2\) and \(n \ge 2d + 2\). Our main contribution is a complete determination of the maximum size of 3-uniform set systems with VC-dimension at most 2 for all values of \(n\). From the perspective of extremal results, this represents a significant first step toward resolving a central problem that has attracted attention from Erd\H{o}s~\cite{1984Erdos}, Frankl and Pach~\cite{1984Franklpach} since the 1980s. In particular, our result answers the specific case of 3-uniform set systems posed by Mubayi and Zhao~\cite{2007JAC} in 2007. We believe that the following research directions are interesting for further exploration.

\subsection{Characterization of extremal configurations for $3$-uniform set systems}
In 2007, Mubayi and Zhao~\cite{2007JAC} extended earlier constructions from~\cite{1997CombFan} and established that there exist infinitely many $(d+1)$-uniform set systems with VC-dimension $d$ and maximum size $\binom{n-1}{d} + \binom{n-4}{d-2}$. They conjectured that this is the correct answer  for~\cref{question}.

We completely resolve the case for $3$-uniform set systems with VC-dimension $2$, proving that for $n \ge 7$, the maximum size is precisely $\binom{n-1}{2} + 1$, thereby confirming their conjecture when $d=2$. An intriguing question that naturally follows is whether the Mubayi-Zhao constructions are the only extremal configurations achieving this bound. Note that the transversal number of their constructions is always 2. However, consider the following 3-uniform set system on $[7]$ with VC-dimension 2:
$$
\begin{aligned}
\big\{ &\{1,2,3\}, \{1,2,4\}, \{1,3,4\}, \{2,3,4\}, \{1,2,5\}, \{1,3,5\}, \{2,3,5\}, \{1,4,5\},\\
&\{1,2,6\}, \{1,3,6\}, \{2,3,6\}, \{2,4,6\}, \{3,5,6\}, \{1,2,7\}, \{1,3,7\}, \{2,3,7\} \big\}.
\end{aligned}
$$
It is straightforward to verify that the transversal number of this system is $3$, demonstrating that even for $n = 7$, the Mubayi-Zhao constructions are not the only extremal configurations achieving the upper bound of $16$. While it is possible to explicitly list all extremal constructions for $n = 7$, it remains an open question whether for larger $n$, there are infinitely many such constructions beyond those provided by Mubayi and Zhao.

\subsection{Improved lower bound when $n=2d$}
Interestingly, we also discovered that the construction previously believed to be optimal fails when \(n = 6\); it only becomes optimal for \(n \ge 7\). This leads to the rather surprising conclusion that the solution to~\cref{question} might naturally divide into three or even more distinct regimes. More precisely, we suspect that such a phenomenon may occur more generally for all \((d+1)\)-uniform set systems with \(d \ge 2\). As supporting evidence, we present a 4-uniform set system \(\mathcal{F} \subseteq \binom{[8]}{4}\) with VC-dimension 3 and size 45, which exceeds the size \(\binom{n-1}{d} + \binom{n-4}{d-2} = 39\) achieved by the previously known constructions in~\cite{1997CombFan,2007JAC}. The system \(\mathcal{F}\subseteq\binom{[8]}{4}\) is given explicitly by:
\[
\begin{aligned}
\big\{ 
&\{1,2,3,4\}, \{1,2,3,5\}, \{1,2,4,5\}, \{1,3,4,5\}, \{2,3,4,5\}, \{1,2,3,6\}, \{1,2,4,6\}, \{1,3,4,6\}, \{2,3,4,6\}, \\
&\{1,2,5,6\}, \{1,3,5,6\}, \{2,3,5,6\}, \{1,4,5,6\}, \{2,4,5,6\}, \{3,4,5,6\}, \{1,2,3,7\}, \{1,2,4,7\}, \{1,3,4,7\}, \\
&\{2,3,4,7\}, \{1,2,5,7\}, \{1,3,5,7\}, \{2,3,5,7\}, \{1,4,5,7\}, \{2,4,5,7\}, \{3,4,5,7\}, \{1,2,6,7\}, \{1,3,6,7\}, \\
&\{2,3,6,7\}, \{1,2,3,8\}, \{1,2,4,8\}, \{1,3,4,8\}, \{2,3,4,8\}, \{1,2,5,8\}, \{1,3,5,8\}, \{2,3,5,8\}, \{1,4,5,8\}, \\
&\{3,4,5,8\}, \{1,4,6,8\}, \{2,4,6,8\}, \{3,4,6,8\}, \{1,5,6,8\}, \{2,5,7,8\}, \{3,5,7,8\}, \{4,5,7,8\}, \{2,4,5,8\}
\big\}.
\end{aligned}
\]
This example further highlights the subtleties in understanding the extremal behavior of set systems under VC-dimension constraints, even for small parameters. In particular, it will be interesting to find a systematical construction for~\cref{question} when $n=2d+2$ that beats Ahlswede-Khachatrian's lower bound. Note that for this particular $n$, the lower bound given by Ahlswede-Khachatrian is
$$\binom{2d + 1}{d} + \binom{2d - 2}{d - 2} = \left(\frac{1}{2} + \frac{1}{16} + o(1)\right)\binom{2d + 2}{d + 1}.$$
We suspect the coefficient $\frac{9}{16}$ can be further improved.

\subsection{Improved upper bound for general $d$}
At present, we think that fully resolving~\cref{question} for all \(d\ge 3\) and all \(n \ge 2d + 2\) is extremely challenging. A potentially attainable goal is to determine whether, for \(d \ge 3\) and sufficiently large $n$, every set system \(\mathcal{F} \subseteq \binom{[n]}{d+1}\) with VC-dimension at most \(d\) has maximum size \(\binom{n-1}{d} + C_d n^{d-2}\) for some constant \(C_d>0\).

\section*{Acknowledgement}
Jian Wang was supported by National Natural Science
Foundation of China Grant no. 12471316.
Zixiang Xu was supported by IBS-R029-C4. Shengtong Zhang was supported by the National Science
Foundation under Grant No. DMS-1928930, while he was in
residence at the Simons Laufer Mathematical Sciences Institute in
Berkeley, California, during the Spring 2025 semester.

\bibliographystyle{abbrv}
\bibliography{FranklPach}

\appendix

\section*{Appendix}
In this appendix, we present the backtracking algorithm used to compute the maximum size of a $3$-uniform set system on $[n]$ with VC-dimension at most $2$, specifically for the cases $n = 6$ and $n = 7$. This algorithm employs an exhaustive yet efficient search strategy, leveraging pruning based on VC-dimension constraints to reduce the search space.

Starting from the full collection of all $\binom{n}{3}$ 3-element subsets of $[n]$, we fix an ordering and denote them by $F_1, F_2, \ldots, F_m$, where $m = \binom{n}{3}$. The algorithm incrementally builds valid set systems by extending candidate collections, while backtracking whenever a potential extension would violate the VC-dimension constraint. The pseudocode is provided below, and the full implementation can be found at~\cite{2025CodesForsmallN}.

\begin{algorithm}[H]
  \KwData{Order all the 3-sets in $\binom{[n]}{3}$ as $F_1,F_2,\ldots,F_m$ where $m=\binom{n}{3}$.}
  \KwResult{Return $M$ as the maximum size of a set system on $[n]$ with VC dimension at most 2. }
  $i_1\leftarrow 1$, $\cF\leftarrow \{F_{i_1}\}$, $j\leftarrow 2$, $\ell\leftarrow i_1$ and $M\leftarrow 1$.
  
  \While{$j\geq 2$}{
    $\ell\leftarrow \ell+1$\;
    \eIf{$\ell>m$}{
      $M \leftarrow \max\{M,j-1\}$\;
      $\ell\leftarrow i_{j-1}$\;
      $j=j-1$\;
    }
    {
    $i_j\leftarrow \ell$\;
    $\cF'\leftarrow \{F_{i_1},F_{i_2},\ldots,F_{i_{j-1}}, F_{i_j}\}$\;      
    \If{ the VC dimension of $\cF'$ is at most 2}{
         $\cF \leftarrow \cF'$\;
         $j\leftarrow j+1$\;
    }
    }
  }
  \caption{The backtracking algorithm}
\end{algorithm}

The algorithm maintains three key variables:
\begin{itemize}
    \item $j$: the current depth (i.e., number of selected sets plus one);
    \item $\ell$: the candidate index for the next 3-set;
    \item $M$: the largest size of a valid set system found so far.
\end{itemize}
The process begins with the first $3$-set $F_1$ and attempts to extend the system by adding further 3-sets in order. At each step, the candidate set $\mathcal{F}'$ is tested for VC-dimension at most $2$. If valid, the set is accepted and the search continues deeper; otherwise, the algorithm skips to the next candidate or backtracks if no further sets remain.

This approach is effective for small $n$, where the search space remains tractable. Our computational experiments show that the algorithm quickly finds the exact maximum sizes:
\begin{itemize}
    \item  For $n = 6$, the maximum is $M = 13$,
\item For $n = 7$, the maximum is $M = 16$.
\end{itemize}
These results constitute a crucial component of our complete resolution of~\cref{question} for $d = 2$.

\end{document}